\documentclass{amsart}[12pt]
\def\doctype{}

\usepackage{latexsym,amssymb,dsfont}
\usepackage{tikz}
\usetikzlibrary { decorations.pathmorphing, decorations.pathreplacing, decorations.shapes, }
\usepackage{color}
\usepackage{fancyhdr}
\usepackage{hyperref}
\hypersetup{
colorlinks=true,
citecolor=blue,
linkcolor=blue,
  urlcolor=blue
}

\newcommand\Z{\mathbb{Z}}

\newcommand\R{\mathbb{R}}

\newcommand{\comment}[1]{}

\newcommand\vb{\mathbf{b}}
\newcommand\vd{\mathbf{d}}
\newcommand\ve{\mathbf{e}}
\newcommand\vj{\mathbf{j}}
\newcommand\vv{\mathbf{v}}
\newcommand\vx{\mathbf{x}}

\newcommand\Sym{\mathcal{S}}

\newcommand\rank{\mathrm{rank}}
\DeclareMathOperator{\sneq}{\cong \hspace{-9pt}\tiny\mbox{$/$} \hspace{3pt}}

\newcommand\posbox[2]{
\filldraw[blue!30!white] (#1,5-#2)--(#1,6-#2)--(#1+1,6-#2)--(#1+1,5-#2);
\node at (#1+0.5,5.5-#2) {\footnotesize $+$};
}
\newcommand\negbox[2]{
\filldraw[red!30!white] (#1,5-#2)--(#1,6-#2)--(#1+1,6-#2)--(#1+1,5-#2);
\node at (#1+0.5,5.5-#2) {\footnotesize $-$};
}
\newcommand\framesix{%
\foreach \a in {0,1,...,6}
\draw (\a,0)--(\a,6);
\foreach \a in {0,6}
\draw[line width=1pt] (\a,0)--(\a,6);
\foreach \a in {0,1,...,6}
\draw (0,\a)--(6,\a);
\foreach \a in {0,6}
\draw[line width=1pt] (0,\a)--(6,\a);
}
\newcommand\rcscope[2]{%
\begin{scope}[scale=0.25]
\foreach \x/\y in {#1}
 \posbox{\x}{\y};
\foreach \x/\y in {#2}
 \negbox{\x}{\y};
\framesix;
\draw[line width=1pt] (3,0)--(3,6);
\draw[line width=1pt] (0,2)--(6,2);
\draw[line width=1pt] (0,4)--(6,4);
\end{scope}
}
\newcommand\rsscope[2]{%
\begin{scope}[xshift=2cm,scale=0.25]
\foreach \x/\y in {#1}
 \posbox{\x}{\y};
\foreach \x/\y in {#2}
 \negbox{\x}{\y};
\framesix;
\draw[line width=1pt] (0,2)--(6,2);
\draw[line width=1pt] (0,4)--(6,4);
\end{scope}
}
\newcommand\csscope[2]{%
\begin{scope}[yshift=-2cm,scale=0.25]
\foreach \x/\y in {#1}
 \posbox{\x}{\y};
\foreach \x/\y in {#2}
 \negbox{\x}{\y};
\framesix;
\draw[line width=1pt] (3,0)--(3,6);
\end{scope}
}
\newcommand\bsscope[2]{%
\begin{scope}[yshift=-2cm,xshift=2cm,scale=0.25]
\foreach \x/\y in {#1}
 \posbox{\x}{\y};
\foreach \x/\y in {#2}
 \negbox{\x}{\y};
\framesix;
\end{scope}
}

\numberwithin{equation}{section}

\fancypagestyle{titlepage}{
\fancyhead[R]{\doctype}
\fancyhead[CL]{}
\cfoot{\vspace{5pt} \thepage}
}

\let\oldsection\section
\newcommand\boldsection[1]{\oldsection{\bf #1}}
\newcommand\starsection[1]{\oldsection*{\bf #1}}
\makeatletter
\renewcommand\section{\@ifstar\starsection\boldsection}
\makeatother

\newtheoremstyle{theorem}
  {12pt}		  % space above
  {0pt}  % space below
  {\sl}  % bofy font
  {\parindent}     % ident - empty=no indent,  \parindent= paragraph indent
  {\bf}  % thm head font
  {. }    % punctuation after thm head
  { }    % space after thm head: `` ``=normal \newline=linebreak
  {}     % thm head specification
\theoremstyle{theorem}
\newtheorem{thm}{Theorem}[section]  % 1st argument is your name for it
\newtheorem{lemma}[thm]{Lemma}     % 2nd argument is what is printed

\newtheorem{prop}[thm]{Proposition}

\newtheoremstyle{definition}
  {12pt}		  % space above
  {0pt}  % space below
  {}  % bofy font
  {\parindent}     % ident - empty=no indent,  \parindent= paragraph indent
  {\bf}  % thm head font
  {. }    % punctuation after thm head
  { }    % space after thm head: `` ``=normal \newline=linebreak
  {}     % thm head specification
\theoremstyle{definition}

\renewcommand{\proofname}{Proof}

\makeatletter
\renewenvironment{proof}[1][\proofname]{\par
  \pushQED{\qed}%
  \normalfont \partopsep=\z@skip \topsep=\z@skip
  \trivlist
  \item[\hskip\labelsep
        \scshape
    #1\@addpunct{.}]\ignorespaces
}{%
  \popQED\endtrivlist\@endpefalse
}
\makeatother

\makeatletter
\renewcommand*\@maketitle{%
  \normalfont\normalsize
  \@adminfootnotes
  \@mkboth{\@nx\shortauthors}{\@nx\shorttitle}%
  \global\topskip42\p@\relax % 5.5pc   "   "   "     "     "
  \@settitle
  \ifx\@empty\authors \else {\vskip 1em
\vtop{\centering\shortauthors\@@par}} \fi
  \ifx\@empty\@date \else {\vskip 1em \vtop{\centering\@date\@@par}}\fi % MYCHANGE
  \ifx\@empty\@dedicatory
  \else
    \baselineskip18\p@
    \vtop{\centering{\footnotesize\itshape\@dedicatory\@@par}%
      \global\dimen@i\prevdepth}\prevdepth\dimen@i
  \fi
  \@setabstract
  \normalsize
  \if@titlepage
    \newpage
  \else
    \dimen@34\p@ \advance\dimen@-\baselineskip
    \vskip\dimen@\relax
  \fi
} % end \@maketitle
\renewcommand*\@adminfootnotes{%
  \let\@makefnmark\relax  \let\@thefnmark\relax
  \ifx\@empty\@subjclass\else \@footnotetext{\@setsubjclass}\fi
  \ifx\@empty\@keywords\else \@footnotetext{\@setkeywords}\fi
  \ifx\@empty\thankses\else \@footnotetext{%
    \def\par{\let\par\@par}\@setthanks}%
  \fi
\thispagestyle{titlepage}
}
\makeatother

\makeatletter \let\c@table\c@figure \makeatother

\begin{document}

\title[]{\large The linear system for Sudoku and a fractional completion threshold}

\author{Peter J.~Dukes and Kate I.~Nimegeers}
\address{
Mathematics and Statistics,
University of Victoria, Victoria, BC, Canada
}
\email{dukes@uvic.ca; nimegeek@uvic.ca}

\date{\today}

\begin{abstract}
We study a system of linear equations associated with Sudoku latin squares.  The coefficient matrix $M$ of the normal system has various symmetries arising from Sudoku.  From this, we find the eigenvalues and eigenvectors of $M$, and compute a generalized inverse.  Then, using linear perturbation methods, we obtain a fractional completion guarantee for sufficiently large and sparse rectangular-box Sudoku puzzles.
\end{abstract}

%\thanks{}

\maketitle
\hrule

\bigskip

\section{Introduction}
\label{sec:intro}

A \emph{latin square} of order $n$ is an $n \times n$ array with entries
from a set of $n$ symbols (often taken to be $[n]:=\{1,2,\dots,n\}$) having
the property that each symbol appears exactly once in every row and every
column.  A \emph{partial latin square} of order $n$ is an $n
\times n$ array whose cells are either empty or filled with one of $n$ symbols in such a way that each symbol appears at most once in every row and every column.
A partial latin square can be identified with a set of ordered triples in a
natural way: if symbol $k$ appears in row $i$ and column $j$, we include the
ordered triple $(i,j,k)$.  A \emph{completion} of a partial latin square $P$
is a latin square $L$ which contains $P$ in the sense of ordered triples;
that is, every symbol occurring in $P$ also occurs in the corresponding cell of $L$.

It is natural to ask how dense a partial latin square can be while still having a completion. Daykin and H\"aggkvist conjectured \cite{DH} that a partial latin square in which any row, column and symbol is used at most $n/4$ times should have a completion.  They proved a weaker version of this claim, with $n/4$ replaced by $2^{-9} n^{1/2}$.  Chetwynd and H\"aggkvist \cite{CH} and later Gustavsson \cite{Gus} obtained the first such completion guarantee which was linear in $n$.  Let us say that a partial latin square is $\epsilon$-dense if no row, column, or symbol is used more than $\epsilon n$ times.  Bartlett \cite{Bartlett} built on the preceding work to show that all $\epsilon$-dense partial latin squares have a completion for $\epsilon = 9.8 \times 10^{-5}$.  
Then, over two papers, this was improved to roughly $\epsilon=0.04$, provided $n$ is large.
One paper \cite{BD} of Bowditch and Dukes obtained this threshold for a fractional relaxation of the problem, and the other paper \cite{BKLOT} by Barber, K\"uhn, Lo, Osthus and Taylor showed using absorbers and balancing graphs that the fractional threshold suffices for very large instances of the (exact) completion problem.

Let $h$ and $w$ be integers with $h,w \ge 2$, and put $n=hw$. A \emph{Sudoku latin square} (or briefly \emph{Sudoku}) of type $(h,w)$ is an $n \times n$ latin square whose cells are partitioned into a $w \times h$ pattern of $h \times w$ subarrays where every symbol appears exactly once in each subarray.  The subarrays are called \emph{boxes}, or sometimes also called \emph{cages}.  A partial Sudoku and completion of such is defined analogously as above for latin squares.  The completion problem for partial Sudoku in the case $h=w=3$  is a famous recreational puzzle.  A mathematical discussion of Sudoku solving strategies can be found in \cite{Sudoku2,Sudoku1}.  By contrast, we are interested here in the fractional relaxation of partial Sudoku completion, essentially following the approach used in \cite{BD} for latin squares.  

Let us explain the fractional relaxation in more detail.  Working from a partial latin square $P$, an empty cell can be assigned a convex combination of symbols instead of a single symbol.  More formally, $P$ can be represented as a function $f_P:[n]^3 \rightarrow \{0,1\}$ in which $f_P(i,j,k)$ is the number of times symbol $k$ appears in cell $(i,j)$.  
A \emph{fractional completion} of $P$ is a function $f:[n]^3 \rightarrow [0,1]$ such that, for any $i,j,k \in [n]$,
\begin{itemize}
\item
$f_P(i,j,k)=1$ implies $f(i,j,k)=1$; and
\item
$\sum_{i=1}^n f(i,j,k) = \sum_{j=1}^n f(i,j,k) = \sum_{k=1}^n f(i,j,k) = 1$.
\end{itemize}
Viewing this as an array, cell $(i,j)$ is assigned a fractional occurrence of symbol $k$ with value $f(i,j,k)$.  The first condition ensures that filled cells of $P$ are left unchanged.  The second condition ensures that every symbol appears with a total occurrence of one in each column and each row, and that every cell is used with a total value of one.
For the Sudoku setting, we can add an extra family of constraints, namely that for all boxes $b$ and symbols $k$, $\sum_{(i,j)\in b} f(i,j,k)=1$, where the sum is over all ordered pairs $(i,j)$ belonging to box $b$.  We remark that when $f$ is $\{0,1\}$-valued, it corresponds to an exact completion of $P$, whether for general latin squares or Sudoku.

Figure~\ref{fig:frac-sud-example} depicts a fractional Sudoku of type $(2,3)$, where the solid disks correspond to pre-filled cells of a partial Sudoku, and the multi-colored disks correspond to a fractional completion.

\begin{figure}[htbp]
\begin{center}
\includegraphics[width=8cm]{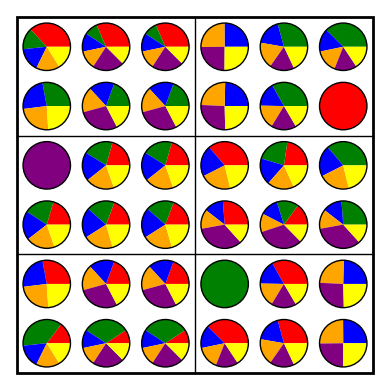}
\caption{Illustration of a fractional completion of a partial Sudoku of type $(2,3)$}
\label{fig:frac-sud-example}
\end{center}
\end{figure}

The notion of $\epsilon$-dense needs to be strengthened in the Sudoku setting.  First, it is natural to impose a constraint on number of filled cells in any box. Otherwise, completion can be blocked by placing symbols $1,\dots,n-1$ in the top-left box and symbol $n$ in line with the remaining empty cell of that box.  This uses each symbol only once and each row and column at most $\max\{h,w\}$ times.  For $h \approx w$, this is sub-linear in $n$. Separately from this, it is also natural to prevent symbol occurrences from being too unbalanced relative to the box partition. 
In more detail, for a Sudoku of type $(h,w)$, it is possible to force a given symbol in, say, the (1,1)-entry by placing it outside of the top-left box in rows $2,\dots,h$ and in columns $2,\dots,w$.  We can arrange for this to occur for different symbols, say $1$ and $2$, making completion impossible.  But this uses no row or column more than twice and uses each symbol only $h+w-2$ times, which could again be sub-linear in $n$. An illustration of the obstructions for $h=w=3$ are given in Figure~\ref{fig:no-completion}.  To address this, we can strengthen the $\epsilon$-dense condition for a partial Sudoku to:
\begin{itemize}
\item
each row, column, and box has at most $\epsilon n$ filled cells;
\item
each symbol occurs at most $\epsilon h$ times in any bundle of $h$ rows corresponding to the box partition, and likewise at most $\epsilon w$ times in any bundle of $w$ columns.  
\end{itemize}

\begin{figure}[htbp]
\begin{tikzpicture}[scale=0.6]
\foreach \a in {0,1,...,9}
	\draw (\a,0)--(\a,9);
\foreach \a in {0,3,6,9}
	\draw[line width=2pt] (\a,0)--(\a,9);
\foreach \a in {0,1,...,9}
	\draw (0,\a)--(9,\a);
\foreach \a in {0,3,6,9}
	\draw[line width=2pt] (0,\a)--(9,\a);
\node at (0.5,8.5) {\small 1};
\node at (1.5,8.5) {\small 2};
\node at (2.5,8.5) {\small 3};
\node at (0.5,7.5) {\small 4};
\node at (1.5,7.5) {\small 5};
\node at (2.5,7.5) {\small 6};
\node at (0.5,6.5) {\small 7};
\node at (1.5,6.5) {\small 8};
\node at (3.5,6.5) {\small 9};
\end{tikzpicture}
\hspace{1cm}
\begin{tikzpicture}[scale=0.6]
\foreach \a in {0,1,...,9}
	\draw (\a,0)--(\a,9);
\foreach \a in {0,3,6,9}
	\draw[line width=2pt] (\a,0)--(\a,9);
\foreach \a in {0,1,...,9}
	\draw (0,\a)--(9,\a);
\foreach \a in {0,3,6,9}
	\draw[line width=2pt] (0,\a)--(9,\a);
\node at (1.5,3.5) {\small 2};
\node at (2.5,0.5) {\small 2};
\node at (1.5,4.5) {\small 1};
\node at (2.5,1.5) {\small 1};
\node at (5.5,7.5) {\small 2};
\node at (8.5,6.5) {\small 2};
\node at (4.5,7.5) {\small 1};
\node at (7.5,6.5) {\small 1};
\end{tikzpicture}
\caption{Sparse partial latin squares with no Sudoku completion}
\label{fig:no-completion}
\end{figure}

Our main result gives a guarantee on fractional completion of $\epsilon$-dense Sudoku latin squares.

\begin{thm}
\label{thm:main}
Let $\epsilon<1/101$.  For sufficiently large integers $h$ and $w$, every $\epsilon$-dense partial Sudoku of type $(h,w)$ has a fractional completion.
\end{thm}

It turns out that our methods to prove Theorem~\ref{thm:main} really only require a weaker notion of density.  Roughly speaking, we need each empty cell $(i,j)$ to have 
a large proportion of all symbols $k$ available to be placed there, and analogous availability when the roles of rows, columns and symbols are permuted, and boxes are introduced.  This is made more precise later.

The outline of the paper is as follows.  In the next section, we reformulate Sudoku completion as a certain graph decomposition problem, and then give a linear system of equations governing the fractional relaxation. Starting with the empty Sudoku, the rank and a basis for the nullspace are computed and interpreted combinatorially.  The $\epsilon$-dense case can be viewed as a perturbed version of the empty case, closely following the approach in \cite{BD} for latin squares.  
This motivates a study of the linear algebra for empty Sudoku in more detail.  
In Section~\ref{sec:aa}, we observe that all relevant computations take place in an adjacency algebra of fixed dimension, independent of $h$ and $w$. 
%Several computations in this algebra are performed with the aid of a computer.  
Eigenvalues and eigenvectors relevant for the problem are described in Section~\ref{sec:eigenvalues}.  Then, using computer-assisted symbolic algebra, a certain generalized inverse is computed and upper-bounded in $\infty$-norm.  This bound ensures solutions to our linear system remain nonnegative.  Section~\ref{sec:perturb} discusses in more detail the perturbation as we pass to the $\epsilon$-dense setting.  By the end of this section, all the ingredients are in place to prove Theorem~\ref{thm:main}.  In Section~\ref{sec:thin}, we consider partial Sudoku completion in the case when the $h \times w$ boxes are asymptotically thin; that is, we consider (say) fixed width $w$ and growing height $h$.  The last section contains some concluding remarks and a brief discussion of possible extensions and generalizations.

\section{Preliminaries and set-up}
\label{sec:setup}

\subsection{A graph decomposition model}
\label{sec:decomp}

In a Sudoku of type $(h,w)$, let $\mathrm{box}(i,j)$ denote the box containing cell $(i,j)$.  If boxes are numbered left to right then top to bottom, then we have the formula
$\mathrm{box}(i,j)=h\lfloor \frac{i-1}{h} \rfloor + \lfloor \frac{j-1}{w} \rfloor + 1$.

We define the graph $G_{hw}$ as follows.  Its vertex set is
$$V(G_{hw}) = \{r_1,\dots,r_n\} \cup \{c_1,\dots,c_n\} \cup\{b_1,\dots,b_n\} \cup\{s_1,\dots,s_n\},$$
with the four sets corresponding to rows, columns, boxes, and symbols, respectively.  Its edge set is
\begin{equation}
\label{eq:EGhw}
E(G_{hw}) = \bigcup_{i,j=1}^n \{\{r_i,c_j\},\{r_i,s_j\}, \{c_i,s_j\}, \{b_i,s_j\}\}.
\end{equation}
In other words, exactly one edge is present for every combination of row-column, row-symbol, column-symbol, and box-symbol.  As a point of notation, $G_{hw}$ depends only on $n=hw$; indeed, if we omit indices in \eqref{eq:EGhw} it is seen to be  simply a blow-up of the graph $K_3+e$ by independent sets of size $n$.  With this in mind, the subscript in $G_{hw}$ can reasonably be interpreted as the product of $h$ and $w$, though it will be useful below to keep these parameters separate. 
Note that the subgraph of $G_{hw}$ induced by rows, columns, and symbols is just the complete $3$-partite graph $K_{n,n,n}$.

A \emph{tile} in $G_{hw}$ is a copy of $K_3+e$ induced by four vertices $r_i,c_j,b_\ell,s_k$ for which $\mathrm{box}(i,j)=\ell$. This tile represents the act of placing symbol $k$ in cell $(i,j)$, and also keeps track of the box used.  Let $T(G_{hw})$ denote the set of all $n^3$ tiles in $G_{hw}$.  

Given a partial Sudoku $S$ of type $(h,w)$, let $G_S$ denote the subgraph of $G_{hw}$ obtained by removing the edge sets of tiles corresponding to filled cells of $S$.  In other words, $V(G_S)=V(G_{hw})$ and $E(G_S)$ contains:
\vspace{-10pt}
\begin{itemize}
\item
$\{r_i,c_j\}$ if and only if cell $(i,j)$ is empty;
\item
$\{r_i,s_k\}$ if and only if symbol $k$ is missing in row $i$;
\item
$\{c_j,s_k\}$ if and only if symbol $k$ is missing in column $j$;
\item
$\{b_\ell,s_k\}$ if and only if symbol $k$ is missing in box $\ell$.
\end{itemize}
Let $T(G_S)$ be the set of tiles in $T(G_{hw})$ all of whose edges are in $G_S$.

An equivalent but slightly different model can be obtained by including row-box and column-box edges.  That is, we could change tiles into cliques $K_4$, and change the host graph $G_{hw}$ into a multigraph $G^*_{hw}$ with the same vertex set and all the edges of $G_{hw}$, and additionally including the edges
\vspace{-10pt}
\begin{itemize}
\item
$\{r_i,b_\ell\}$ with multiplicity $w$ if and only if row $i$ is incident with box $\ell$; and
\item
$\{c_j,b_\ell\}$ with multiplicity $h$ if and only if column $j$ is incident with box $\ell$.
\end{itemize}
Likewise, given a partial Sudoku $S$, we could define $G_S^*$ to be the subgraph of $G_{hw}^*$ obtained by removing the edges of all $4$-cliques on $\{r_i,c_j,b_\ell,s_k\}$ whenever symbol $k$ occurs in cell $(i,j)$ (and box $\ell)$ of $S$.

For graphs $F$ and $G$, we say that $G$ has an $F$-\emph{decomposition} if its edge set $E(G)$ can be partitioned into subgraphs isomorphic to $F$.  This extends naturally to multigraphs $G$, where now repeated edges are distinguished.  That is, the number of copies of $F$ containing two vertices $u \neq v$ equals the multiplicity of edge $\{u,v\}$ in $G$.  Many problems in combinatorics can be formulated in terms of graph decompositions.  For example, $K_3$-decompositions of $K_{n,n,n}$ are equivalent to latin squares of order $n$; see for instance \cite{BKLOT,BD,Montgomery}.  The following is an analog for Sudoku using our graphs above.

\begin{prop}
\label{prop:sud-decomp}
The partial Sudoku $S$ has a completion if and only if the graph $G_S$ has an edge-decomposition into tiles, or equivalently if and only if $G^*_S$ has a $K_4$-decomposition.
\end{prop}

\begin{proof}
Suppose $S$ has a completion $S'$.  If cell $(i,j)$ was blank in $S$ but filled in $S'$, say with symbol $k$, we use the tile defined by $\{r_i,c_j,s_k,b_\ell\}$, where $\ell=\mathrm{box}(i,j)$.  Each such tile belongs to $T(G_S)$ because $(i,j)$ was blank in $S$ and because $k$ occurs only once in $S'$ in row $i$, column $j$ and box $\ell$.  Consider the set $\mathcal{T}$ of these tiles induced by cells that were blank in $S$ and filled in $S'$.  These tiles are edge-disjoint, again because $S'$ has no repeated symbols in any any row, column, or box.  We check that $\mathcal{T}$ gives an edge-decomposition of $G_S$ into tiles. Any row-column edge of $G_S$, say $\{r_i,c_j\}$, is in the tile corresponding to the symbol placed at entry $(i,j)$ of $S'$.  Consider a row-symbol edge, say $\{r_i,s_k\} \in E(G_S)$.  The presence of this edge means $k$ was missing from row $i$ in $S$.  It occurs somewhere in row $i$ of $S'$, say at entry $(i,j)$.  This entry was blank in $S$, so $r_i,c_j,s_k$ define a tile in $\mathcal{T}$, along with the box containing $(i,j)$.  A similar verification holds for edges of type column-symbol and box-symbol in $G_S$.

For the converse, the argument is reversible.  Given a set $\mathcal{T}$ of tiles that form an edge-decomposition of $G_S$, we complete $S$ by placing symbol $k$ in entry $(i,j)$ whenever $r_i,c_j,s_k$ belong to a tile of $\mathcal{T}$.  Since the tiles of $\mathcal{T}$ are edge-disjoint, every entry $(i,j)$ is filled at most once and no row, column, or box contains repeats.  Since the edges within $\mathcal{T}$ partition those in $G_S$, it follows that every blank entry of $S$ gets filled, and every symbol occurs in every row, column, and box.

The claim about $G_S^*$ having a $K_4$-decomposition is nearly identical.  In the forward implication, we note that a row-box edge $\{r_i,b_\ell\}$ in $G_{hw}^*$ occurs in total $w$ times counting $E(G_S^*)$ and the decomposition.  This is because the completion $S'$ has $w$ entries in row $i$ and box $\ell$.  Similarly, column-box edges of $G_{hw}^*$ occur a total of $h$ times.
\end{proof}

The model using row-box and column-box edges has the advantage that all $4$-cliques in $G_S^*$ correspond to valid tiles.  However, since no new information is carried by those extra edges, we henceforth work with tiles in $G_S$, omitting the implied edges of type row-box and column-box.

This paper is concerned with partial Sudoku latin squares which are nearly empty. Recall the definition of $\epsilon$-dense discussed in Section~\ref{sec:intro} and strengthened for the case of Sudoku.  The definition leads easily to various degree bounds in $G_S$, which we summarize here.

\begin{lemma}
\label{lem:eps-GS}
Suppose $S$ is $\epsilon$-dense.  
Then in the graph $G_S$, the number of edges from vertex
\vspace{-11pt}
\begin{itemize}
\item
$c_j$ to the row partite set is at least $(1-\epsilon)n$;
\item
$s_k$ to any row bundle is at least $(1-\epsilon)h$;
\item
$r_i$ to the column partite set is at least $(1-\epsilon)n$;
\item
$s_k$ to any column bundle is at least $(1-\epsilon)w$;
\item
$r_i,c_j$ or $b_\ell$ to the symbol partite set is at least $(1-\epsilon)n$;
\item
$s_k$ to the box partite set is at least $(1-\epsilon)n$.
\end{itemize}
\end{lemma}

Next, we give an alternate sparseness definition suited to our approach.
Let us say that a partial Sudoku $S$ has the $(1-\delta)$-\emph{availability property} if every edge $e \in E(G_S)$ is contained in at least $(1-\delta)n$ tiles in $T(G_S)$. 
We note that an $\epsilon$-dense partial Sudoku has the $(1-3\epsilon)$-availability property.  Indeed, for an edge $\{r_i,c_j\}$, at most $\epsilon n$ symbols are already used in row $i$, in column $j$, and in the box containing $(i,j)$.  These could all be different sets of symbols, but even so this leaves at least $(1-3\epsilon)n$ available tiles.
For an edge of the form $\{r_i,s_k\}$, at most $\epsilon n$ columns are filled in row $i$, at most $\epsilon n$ columns already have symbol $k$, and at most $(\epsilon h)w=\epsilon n$ columns are unavailable due to the boxes along row $i$ already having symbol $k$.  Edges $\{c_j,s_k\}$ behave analogously.  Finally, an edge of type $\{s_k,b_\ell\}$ has at most $\epsilon n$ unavailable options due to filled cells in box $\ell$, and at most $\epsilon(h+w)$ unavailable options due to symbol $k$ occurring elsewhere in the row bundle or column bundle for box $\ell$.
For each of the four edge types, we have at least $(1-3\epsilon)n$ available tiles in $T(G_S)$.

Conversely, $S$ having the $(1-\delta)$-availability property implies any row, column, symbol or box that is not completely used has at most $\delta n$ occurrences.  Moreover, if symbol $s_k$ occurs fewer than $h$ times in a row bundle, it must occur at most $\delta h$ times.  Otherwise, take a box $b_\ell$ in this bundle with empty cells and observe that $\{s_k,b_\ell\}$ has fewer than $n-w(\delta h)=(1-\delta)n$ tiles available in $T(G_S)$.  A similar statement holds for column bundles.  In other words, the $(1-\delta)$-availability property implies $S$ is $\delta$-dense, except possibly for completely filled rows, columns, boxes, or any symbols fully used in a bundle.  This exception is convenient if, say, one wants to finish off certain symbols in a bundle before completing the rest of the partial Sudoku.  This idea is revisited in Section~\ref{sec:thin}.

\subsection{The linear system for Sudoku}
\label{sec:linear-system}

Consider an empty $n \times n$ Sudoku, to be filled with entries from $[n]$.
Let $x_{ijk}$ denote the number/fraction of symbols $s_k$ placed in cell $(i,j)$, where $(i,j,k) \in [n]^3$.  Latin square and Sudoku constraints naturally correspond to linear equations on these variables.
The condition that every cell have exactly one entry becomes $\sum_{k} x_{ijk} = 1$ for each $(i,j) \in [n]^2$.
The condition that every row contains every symbol exactly once becomes $\sum_{j} x_{ijk} = 1$ for each $(i,k) \in [n]^2$.  Similarly, that every column contains every symbol exactly once becomes $\sum_{i} x_{ijk} = 1$ for each $(j,k) \in [n]^2$.
Together, these $3n^2$ equations yield a linear system for (fractional) latin squares.  The additional condition relevant for Sudoku is that every box contains every symbol exactly once, or 
$$\sum_{(i,j) \in \mathrm{box}(\ell)} x_{ijk} = 1$$ for each $(k,\ell) \in [n]^2$.

This results in a $4n^2 \times n^3$ linear system
\begin{equation}
\label{eq:Wsys}
W \mathbf{x}= \mathds{1},
\end{equation}
where $\mathds{1}$ denotes the all-ones vector and $W$ is the $\{0,1\}$ inclusion matrix of $E(G_{hw})$ versus $T(G_{hw})$; that is, $W(e,t)=1$ if $e \in t$ and is $0$ otherwise. In this paper, we will mainly consider the (square) normal system, with coefficient matrix $M=WW^\top$.  An entrywise nonnegative solution $\mathbf{y}$ to $M \mathbf{y}=\mathds{1}$ implies the existence of a solution $\mathbf{x}=W^\top \mathbf{y} \ge \mathbf{0}$ to \eqref{eq:Wsys}.  This, in turn, produces a fractional edge-decomposition of $G_{hw}$ into tiles and a fractional Sudoku of type $(h,w)$.

The rank (over the reals) of both $W$ and $M$ can be found by exhibiting a basis for their range consisting of tiles in $G_{hw}$.
For convenience, set punctuation will be omitted from edges and tiles; we abbreviate these by juxtaposing vertices in different partite sets of $V(G_{hw})$.

\begin{prop}
\label{prop:rank}
We have $\rank(M)=\rank(W)=n^3-(n-1)^3+(n-1)(h-1)(w-1)$.
\end{prop}

\begin{proof}
Let $\mathcal{T}_1$ be the set of $n^3-(n-1)^3$ tiles in $G_{hw}$ which intersect at least one of $r_1,c_1,s_1$.  It was shown in \cite[Proposition 2.3]{BD} that $\mathcal{T}_1$ is linearly independent in the vector space of functions from $E(K_{n,n,n})$ to $\R$. Thus, $\mathcal{T}_1$ is also linearly independent in $\R^{E(G_{hw})}$.  Let $\mathcal{T}_2$ be any set of $(n-1)(h-1)(w-1)$ tiles of the form $r_i c_j s_k b_\ell$, where $k=2,\dots,n-1$ and the $b_\ell$ range over all boxes disjoint from both row $1$ and column $1$.  Since the tiles in $\mathcal{T}_2$  use distinct box-symbol edges which are not present in $\mathcal{T}_1$, it is clear that $\mathcal{T}_1 \cup \mathcal{T}_2$ is linearly independent.

We next show that $\mathcal{T}_1 \cup \mathcal{T}_2$ generates any given column of $W$, say the one corresponding to a tile $\{r_i,c_j,s_k,b_\ell\}$.
Suppose $i\le h$ and $j \le w$.  Then $\ell=\mathrm{box}(i,j)=1$.  We can form the target tile as a linear combination in $\mathcal{T}_1$, namely as
$$r_ic_js_kb_1
= r_1c_js_kb_1
+ r_ic_1s_kb_1
+ r_ic_js_1b_1
- r_1c_1s_kb_1
- r_1c_js_1b_1
- r_ic_1s_1b_1
+ r_1c_1s_1b_1.$$
Suppose next that $i>h$ and $j \le w$.  
As above, we have the linear combination
$$r_ic_js_kb_\ell
= r_1c_js_kb_1
+ r_ic_1s_kb_\ell
+ r_ic_js_1b_\ell
- r_1c_1s_kb_1
- r_1c_js_1b_\ell
- r_ic_1s_1b_1
+ r_1c_1s_1b_1.$$
Similarly, $\mathcal{T}_1$ generates any tile with $i\le h$ and $j>w$.  Suppose, then, that $i>h$ and $j>w$.  If $k=1$, the corresponding tile belongs to $\mathcal{T}_1$, so assume $k>1$.  Put $p=\mathrm{box}(1,j)$ and $q=\mathrm{box}(i,1)$, and note that all tiles meeting these boxes are in the span of $\mathcal{T}_1$, as shown above.  Since $i>h$, $j>w$, and $k>1$, we know that
$s_k$ and $b_\ell$ occur together in some tile $r_{i'} c_{j'} s_k b_\ell \in \mathcal{T}_2$.  Using this and other tiles generated so far, we compute
\begin{align*}
r_ic_js_kb_\ell
&=
r_ic_js_1b_\ell
+r_1 c_j s_k b_p
+r_i c_1 s_k b_q
-r_1 c_1 s_k b_1
-r_1 c_j s_1 b_p
-r_i c_1 s_k b_q
+r_{i'} c_{j'} s_k b_\ell\\
&-r_{i'} c_{j'} s_1 b_\ell
-r_{1} c_{j'} s_k b_p
-r_{i'} c_1 s_k b_q
+r_1 c_1 s_k b_1
+r_1 c_{j'} s_1 b_p
+r_{i'} c_1 s_1 b_q
-r_1 c_1 s_1 b_1.
\end{align*}
We have shown that $\mathcal{T}_1 \cup \mathcal{T}_2$ spans each column of $W$, and hence is a basis for range$(W)$.
\end{proof}

Suppose now that some cells of our Sudoku have been pre-filled, resulting in the graph $G_S$ as described earlier.  Let $W_S$ denote the $\{0,1\}$ inclusion matrix of edges versus tiles in $G_S$.  Then the system
\begin{equation}
\label{eq:WSsys}
W_S \mathbf{x}= \mathds{1}
\end{equation}
has a solution $\mathbf{x} \ge \mathbf{0}$ if and only if $G_S$ admits a fractional edge-decomposition into tiles. Note that for non-empty $S$, the dimensions in the system \eqref{eq:WSsys} are smaller than in \eqref{eq:Wsys}.  The tile weights are given by entries $x_{ijk}$ of $\mathbf{x}$.  By Proposition~\ref{prop:sud-decomp}, the existence of such a solution is equivalent to our partial Sudoku $S$ having a completion.

We may again consider the normal system with coefficient matrix $M_S=W_S W_S^\top$.  Even though many possible solutions of \eqref{eq:WSsys} are lost in doing so, the normal system has the advantage of allowing eigenvalue and perturbation methods, as was done in \cite{BD}.  We extend these methods to the Sudoku setting in Sections~\ref{sec:spectral} and \ref{sec:perturb} to follow.

\subsection{The kernel}
\label{sec:kernel}

For our analysis of the linear systems \eqref{eq:Wsys} and \eqref{eq:WSsys} above, it is important to study the nullspace/kernel of $M$, or equivalently the left nullspace of $W$.  This can be viewed as the set of all edge-weightings in $G_{hw}$ in which each tile has a vanishing total weight (over its four edges).

After some simplification, Proposition~\ref{prop:rank} gives
\begin{equation}
\label{eq:nullity}
\dim \, \ker(M)=4n^2-\rank(M) = 3n+(h+w)(n-1).
\end{equation}
We would like to find a basis for $\ker(M)$.
First, a spanning set is described in three categories of vectors below.

(A) Choose a row $r_i$ and define the vector $\mathbf{v}$, coordinates indexed by $E(G_{hw})$, where $\mathbf{v}(r_ic_j)=1$ for all $j \in [n]$, $\mathbf{v}(r_is_k)=-1$ for all $k \in [n]$, and otherwise $\mathbf{v}(e)=0$.   Similar classes of kernel vectors exist with the roles of row, column and symbol permuted.

Consider the characteristic vector $\mathbf{t}$ of some tile $t$.  If $r_i \in t$, then since $t$ contains exactly one column and exactly one symbol we have $\mathbf{t} \cdot \mathbf{v}=1-1=0$.  On the other hand, if $r_i \not\in t$, the support of $t$ is disjoint from the support of $\mathbf{v}$, hence we again have $\mathbf{t} \cdot \mathbf{v}=0$.
In plain language, this kernel vector encodes the condition that the number of times a column is used with row $i$ equals the number of times a symbol is used in row $i$. Verification is similar for the permuted varieties.

(B) Choose a box $b_\ell$ and define the vector $\mathbf{v}$ in which $\mathbf{v}(r_ic_j)=1$ for all $(i,j) \in b_\ell$, $\mathbf{v}(b_\ell s_k)=-1$ for all $k \in [n]$, and otherwise $\mathbf{v}(e)=0$. 

As before, let $\mathbf{t}$ be the characteristic vector of a tile $t$.  If $b_\ell \in t$, then since $t$ contains exactly one row, column and symbol, we have $\mathbf{t} \cdot \mathbf{v}=1-1=0$.  On the other hand, if $b_\ell \not\in t$, the supports of $\mathbf{t}$ and $\mathbf{v}$ are disjoint.  This kernel vector encodes the condition that the number of entries filled in box $\ell$ equals the number of symbols used in box $\ell$.

(C) Choose a bundle of rows $\{r_{hp+1},\dots,r_{h(p+1)}\}$ and a
symbol $s_k$.  Define the vector $\mathbf{v}$ with $\mathbf{v}(r_is_k)=1$ and $\mathbf{v}(b_\ell s_k)=-1$
for all $hp < i,\ell \le h(p+1)$, and otherwise $\mathbf{v}(e)=0$.  A similar class of vectors exists for column bundles.

Once again, let $\mathbf{t}$ be the characteristic vector of a tile $t$.  Suppose $s_k \in t$ and $t$ intersects the row bundle defining $\mathbf{v}$.  Since $t$ contains exactly one row and exactly one box meeting this bundle, we have $\mathbf{t} \cdot \mathbf{v}=1-1=0$.  On the other hand, if either $s_k \not\in t$ or $t$ intersects a different row bundle, the supports of $\mathbf{t}$ and $\mathbf{v}$ are disjoint.  The encoded condition states that a given symbol $s_k$ appears the same number of times in a row bundle as in the corresponding box bundle.  The column bundle case has a similar verification.

To see that the vectors above span $\ker(M)$, it is convenient to examine the subspace $U$ spanned by the row and column varieties of type (A), along with all vectors of type (B).  Note that all vectors in $U$ are invariant under symbol permutation.  It is straightforward to see that $\dim(U)=3n$, since basis vectors have disjoint supports.  Next, let $V$ be the subspace of $\ker(M)$ spanned by the type (C) vectors.  It is easy to see that $\dim(V)=(h+w)n$ and $\dim(U \cap V) \le h+w$, the number of row or column bundles.  Therefore, $$\dim(U+V) = \dim(U)+\dim(V) - \dim(U \cap V) \ge 3n+(h+w)(n-1) = \dim \ker(M).$$
For a basis of $\ker(M)$, it suffices to take the row and column varieties of type (A) vectors, all type (B) vectors, and those of type (C) which avoid a particular symbol, say $s_n$.  More details on this can be found in the second author's thesis \cite{Kate}.

Let $K$ be the $4n^2 \times 4n^2$ matrix which projects onto $\ker(M)$.  We have $K^2 = K = K^\top$.  Let $K[S]$ denote the principal submatrix of $K$ whose rows and columns correspond to the edges of $G_S$. The following orthogonality relations are similar to \cite[Proposition 2.5]{BD}.

\begin{prop}
\label{prop:Kres-MG}
The range of $K[S]$ is orthogonal to the all-ones vector and to the range of $M_S$.  That is,
{\rm (a)}
$K[S] \mathds{1} = \mathbf{0}$; and
{\rm (b)}
$K[S] M_S = O$.
\end{prop}

\begin{proof}
For matrices and vectors indexed by $E(G_{hw})$, sort the indices so that those corresponding to $E(G_S)$ come first.  Let $L$ be the inclusion map from edges of $G_S$ to edges of $G_{hw}$ and let $Q$ be the inclusion map from tiles of $G_S$ to tiles of $G_{hw}$.  As matrices, $L$ and $Q$ have the structure $[I \mid O]^\top$.  

Let $\mathds{1}_S = (\mathds{1} \mid \mathbf{0})$ be the $4n^2 \times 1$ zero-one indicator vector of $E(G_S)$ in $E(G_{hw})$. 
Alternatively, $\mathds{1}_S$ is obtained from the
$4n^2 \times 1$ all-ones vector by subtracting indicator vectors of tiles in $S$.  It follows that $\mathds{1}_S$ is contained in the range of $W^\top$, and hence is orthogonal to $\mathrm{ker}(W^\top)$.  We now compute
$$K[S] \mathds{1} = L^\top KL \mathds{1} =  L^\top K \mathds{1}_S = \mathbf{0}.$$ 
This proves (a).  With our matrix partition, we have
$$W=\left[
\begin{array}{c|c}
W_S & *\\
\hline
O & *
\end{array} \right]$$
and $LW_S=WQ$.  Working from these,
$$K[S] M_S = L^\top K L W_S W_S^\top = L^\top K W Q W_S^\top = O,$$
since $KW=O$.  This proves (b).
\end{proof}

Next, we recall \cite[Lemma 2.6]{BD}. The idea lets us solve an under-determined system $A \vx = \vb$ by inverting an additive shift of $A$.  We use this later in Section~\ref{sec:perturb} with $A$ taking the role of $M_S$, $B$ a multiple of $K[S]$, and $\vb = \mathds{1}$.

\begin{lemma}[see \cite{BD}]
\label{lem:shift}
Let $A$ and $B$ be symmetric $N \times N$ real matrices with $AB=O$, $A+B$
nonsingular, and $B \vb = \mathbf{0}$. Then $A(A+B)^{-1} \vb = \vb$.
\end{lemma}

\section{The adjacency algebra}
\label{sec:aa}

\subsection{A coherent configuration for Sudoku}

Let $X$ be a finite set.
A \emph{coherent configuration} on $X$ is a partition of $X \times X$ into a set of relations $\mathcal{R}=\{R_1,\dots,R_d\}$ satisfying the following properties:
\begin{enumerate}
\item
the union of some relations in $\mathcal{R}$ equals the diagonal $\{(x,x):x \in X\}$;
\item
for each $R$ in $\mathcal{R}$, the transpose relation $R^\top =\{(y,x):(x,y) \in R\}$ is also in $\mathcal{R}$;
\item
for each $i,j,k$, there exists a constant $p_{ij}^k$ such that for any $x,z$ with $(x,z) \in R_k$, there are exactly $p_{ij}^k$ elements $y$ such that $(x,y) \in R_i$ and $(y,z) \in R_j$.
\end{enumerate}

Given a group $G$ acting on a set $X$, the set of orbits of the induced action on $X \times X$ is a coherent configuration, \cite{Higman}.  Here, we set up a coherent configuration on the ground set $X=E(G_{hw})$ using a group of Sudoku symmetries.  
The wreath product $\Sym_h \wr \Sym_w$ acts on rows and preserves the row bundle partition.  Similarly, $\Sym_w \wr \Sym_h$ acts on columns.  The direct product of these acts on $[n]^2$ and in particular the second factors $\Sym_w \times \Sym_h$ act on the $n$ boxes.  Finally, if we take a product with $\Sym_n$ for symbol symmetries, we have the group
$$\Gamma_{hw}=(\Sym_h \wr \Sym_w) \times (\Sym_w \wr \Sym_h) \times \Sym_n \le \mathrm{Aut}(G_{hw})$$
acting on $E(G_{hw}) = \{r_i c_j, r_is_k, c_j s_k, b_\ell s_k: i,j,k,\ell \in [n]\}$.

We describe the relations on ordered pairs of edges in more detail.
Given two rows $r_i$ and $r_{i'}$, we write $r_i \sim r_{i'}$ if and only if 
$\lfloor (i-1)/h \rfloor = \lfloor (i'-1)/h \rfloor$.  Similarly, given two columns $c_j$ and $c_{j'}$, we write $c_j \sim c_{j'}$ iff
$\lfloor (j-1)/w \rfloor = \lfloor (j'-1)/w \rfloor$.  In other words, $\sim$ tracks whether two rows or two columns belong to the same bundle. From the definition, it is clear that $\sim$ is an equivalence relation on both rows and columns.
Write $r_i \sneq r_{i'}$ if $r_i \sim r_{i'}$ but $r_i \neq r_{i'}$.  Define $\sneq$ similarly for columns.

Given two boxes $b_{\ell}$ and $b_{\ell'}$, write $b_\ell \smile b_{\ell'}$ if $\lfloor (\ell-1)/h \rfloor = \lfloor (\ell'-1)/h \rfloor$.  Informally, this keeps track of whether the two boxes occur in the same row bundle.  Write $b_\ell \frown b_{\ell'}$ iff $\ell \equiv \ell' \pmod{h}$; this is the analog for boxes in the same column bundle.
By abuse of notation, we write $r_i \smile b_{\ell}$ to mean that the corresponding row and box intersect, and similarly for $c_j \frown b_{\ell}$.

Ordered pairs of vertices in $G_{hw}$ are partitioned into relations according to Table~\ref{tab:vertex-rels}.  A blank indicates the trivial relation.
Moving from vertices to edges, there are $69$ relations induced on ordered pairs of edges.  These are labelled and displayed in Figure~\ref{fig:edge-rels}. 

\begin{table}[htbp]
\begin{tabular}{c|cccc}
& rows & cols & symbols & boxes \\
\hline
rows & $=,\sneq,\nsim$ & & & $\smile,\not\smile$\\
cols & & $=,\sneq,\nsim$ & & $\frown,\not\frown$\\
symbols & & & $=,\neq$ & \\
boxes & $\smile,\not\smile$ & $\frown,\not\frown$ & & $=,\smile,\frown,\not\asymp$ \\
\end{tabular}
\caption{Relations on vertices of $G_{hw}$}
\label{tab:vertex-rels}
\end{table}

\begin{figure}[htbp]
\begin{tikzpicture}[scale=0.6]
\foreach \a in {0,6,12,18,24}
	\draw[line width=2pt] (0,\a)--(24,\a);
\foreach \a in {0,6,12,18,24}
	\draw[line width=2pt] (\a,0)--(\a,24);

\node at (-0.6,3) {\rotatebox{90}{box-symbol}};
\node at (-0.6,9) {\rotatebox{90}{col-symbol}};
\node at (-0.6,15) {\rotatebox{90}{row-symbol}};
\node at (-0.6,21) {\rotatebox{90}{row-col}};
\node at (3,24.6) {row-col};
\node at (9,24.6) {row-symbol};
\node at (15,24.6) {col-symbol};
\node at (21,24.6) {box-symbol};

\begin{scope}[xshift=0cm,yshift=18cm]
\draw (2,0)--(2,6);
\draw (4,0)--(4,6);
\draw (0,2)--(6,2);
\draw (0,4)--(6,4);
\node[align=center] at (1,5) {\small $r_i=r_{i'}$\\ \small $c_j=c_{j'}$};
\node[align=center] at (3,5) {\small $r_i=r_{i'}$\\ \small $c_j \sneq c_{j'}$};
\node[align=center] at (5,5) {\small $r_i=r_{i'}$\\ \small $c_j \nsim c_{j'}$};

\node[align=center] at (1,3) {\small $r_i \sneq r_{i'}$\\ \small $c_j=c_{j'}$};
\node[align=center] at (3,3) {\small $r_i \sneq r_{i'}$\\ \small $c_j \sneq c_{j'}$};
\node[align=center] at (5,3) {\small $r_i \sneq r_{i'}$\\ \small $c_j \nsim c_{j'}$};

\node[align=center] at (1,1) {$r_i \nsim r_{i'}$\\ \small $c_j=c_{j'}$};
\node[align=center] at (3,1) {$r_i \nsim r_{i'}$\\ \small $c_j \sneq c_{j'}$};
\node[align=center] at (5,1) {$r_i \nsim r_{i'}$\\ \small $c_j \nsim c_{j'}$};
\node[above right=-2pt] at (0,4) {\footnotesize \tt 1};
\node[above right=-2pt] at (2,4) {\footnotesize \tt 2};
\node[above right=-2pt] at (4,4) {\footnotesize \tt 3};
\node[above right=-2pt] at (0,2) {\footnotesize \tt 4};
\node[above right=-2pt] at (2,2) {\footnotesize \tt 5};
\node[above right=-2pt] at (4,2) {\footnotesize \tt 6};
\node[above right=-2pt] at (0,0) {\footnotesize \tt 7};
\node[above right=-2pt] at (2,0) {\footnotesize \tt 8};
\node[above right=-2pt] at (4,0) {\footnotesize \tt 9};
\filldraw[color=yellow,opacity=0.05](0,0) rectangle (6,6);
\end{scope}

\begin{scope}[xshift=6cm,yshift=18cm]
\draw (0,2)--(6,2);
\draw (0,4)--(6,4);
\node[align=center] at (3,5) {$r_i=r_{i'}$};
\node[align=center] at (3,3) {$r_i \sneq r_{i'}$};
\node[align=center] at (3,1) {$r_i \nsim r_{i'}$};
\node[above right=-2pt] at (0,4) {\footnotesize \tt 10};
\node[above right=-2pt] at (0,2) {\footnotesize \tt 11};
\node[above right=-2pt] at (0,0) {\footnotesize \tt 12};
\filldraw[color=red,opacity=0.05](0,0) rectangle (6,6);
\end{scope}

\begin{scope}[xshift=12cm,yshift=18cm]
\draw (2,0)--(2,6);
\draw (4,0)--(4,6);
\node[align=center] at (1,3) {\small $c_j=c_{j'}$};
\node[align=center] at (3,3) {\small $c_j \sneq c_{j'}$};
\node[align=center] at (5,3) {\small $c_j \nsim c_{j'}$};
\node[above right=-2pt] at (0,0) {\footnotesize \tt 22};
\node[above right=-2pt] at (2,0) {\footnotesize \tt 23};
\node[above right=-2pt] at (4,0) {\footnotesize \tt 24};
\filldraw[color=green,opacity=0.05](0,0) rectangle (6,6);
\end{scope}

\begin{scope}[xshift=18cm,yshift=18cm]
\draw[decorate,decoration={coil,aspect=0,amplitude=1pt}] (3,0)--(3,6);
\draw[decorate,decoration={coil,aspect=0,amplitude=1pt}] (0,3)--(6,3);
\node[align=center] at (1.5,4.5) {$b_\ell$\\ $\asymp$ \\$b_{\text{box}(i,j)}$};
\node[align=center] at (4.5,4.5) {$b_\ell$\\ $\smile$ \\$b_{\text{box}(i,j)}$};
\node[align=center] at (1.5,1.5) {$b_\ell$\\ $\frown$ \\$b_{\text{box}(i,j)}$};
\node[align=center] at (4.5,1.5) {$b_\ell$\\ $\not\asymp$ \\$b_{\text{box}(i,j)}$};
\node[above right=-2pt] at (0,3) {\footnotesize \tt 38};
\node[above right=-2pt] at (3,3) {\footnotesize \tt 39};
\node[above right=-2pt] at (0,0) {\footnotesize \tt 40};
\node[above right=-2pt] at (3,0) {\footnotesize \tt 41};
\filldraw[color=orange,opacity=0.05](0,0) rectangle (6,6);
\end{scope}

\begin{scope}[xshift=0cm,yshift=12cm]
\draw (0,2)--(6,2);
\draw (0,4)--(6,4);
\node[align=center] at (3,5) {$r_i=r_{i'}$};
\node[align=center] at (3,3) {$r_i \sneq r_{i'}$};
\node[align=center] at (3,1) {$r_i \nsim r_{i'}$};
\node[above right=-2pt] at (0,4) {\footnotesize \tt 13};
\node[above right=-2pt] at (0,2) {\footnotesize \tt 14};
\node[above right=-2pt] at (0,0) {\footnotesize \tt 15};
\filldraw[color=red,opacity=0.05](0,0) rectangle (6,6);
\end{scope}

\begin{scope}[xshift=6cm,yshift=12cm]
\draw (0,2)--(6,2);
\draw (0,4)--(6,4);
\draw (2,0)--(4,2);
\draw (2,2)--(4,4);
\draw (2,4)--(4,6);
\node[align=center] at (1.3,5) {$r_i=r_{i'}$\\$s_k=s_{k'}$};
\node[align=center] at (1.3,3) {$r_i \sneq r_{i'}$\\$s_k=s_{k'}$};
\node[align=center] at (1.3,1) {$r_i \nsim r_{i'}$\\$s_k=s_{k'}$};
\node[align=center] at (4.6,5) {$r_i=r_{i'}$\\$s_k \neq s_{k'}$};
\node[align=center] at (4.6,3) {$r_i \sneq r_{i'}$\\$s_k \neq s_{k'}$};
\node[align=center] at (4.6,1) {$r_i \nsim r_{i'}$\\$s_k \neq s_{k'}$};
\node[above=-2pt] at (3,5.5) {\footnotesize \tt 16};
\node[below=-2pt] at (3,4.5) {\footnotesize \tt 17};
\node[above=-2pt] at (3,3.5) {\footnotesize \tt 18};
\node[below=-2pt] at (3,2.5) {\footnotesize \tt 19};
\node[above=-2pt] at (3,1.5) {\footnotesize \tt 20};
\node[below=-2pt] at (3,0.5) {\footnotesize \tt 21};
\filldraw[color=magenta,opacity=0.05](0,0) rectangle (6,6);
\end{scope}

\begin{scope}[xshift=12cm,yshift=12cm]
\draw (0,0)--(6,6);
\node[align=center] at (2,4) {\large $s_k=s_{k'}$};
\node[align=center] at (4,2) {\large $s_k \neq s_{k'}$};
\node[right=-2pt] at (0,1) {\footnotesize \tt 28};
\node[above=-2pt] at (1,0) {\footnotesize \tt 29};
\filldraw[color=blue,opacity=0.05](0,0) rectangle (6,6);
\end{scope}

\begin{scope}[xshift=18cm,yshift=12cm]
\draw[decorate,decoration={coil,aspect=0,amplitude=1pt}] (0,3)--(6,3);
\draw (1.5,0)--(4.5,3);
\draw (1.5,3)--(4.5,6);
\node[align=center] at (1.5,4.7) {$\smile$\\$s_k=s_{k'}$};
\node[align=center] at (4.5,4.3) {$\smile$\\$s_k \neq s_{k'}$};
\node[align=center] at (1.5,1.7) {$\not\smile$\\$s_k=s_{k'}$};
\node[align=center] at (4.5,1.3) {$\not\smile$\\$s_k\neq s_{k'}$};
\node[above right=-2pt] at (0,3) {\footnotesize \tt 46};
\node[above right=-2pt] at (2,3) {\footnotesize \tt 47};
\node[above right=-2pt] at (0,0) {\footnotesize \tt 48};
\node[above right=-2pt] at (2,0) {\footnotesize \tt 49};
\filldraw[color=black!50!red,opacity=0.05](0,0) rectangle (6,6);
\end{scope}

\begin{scope}[xshift=0cm,yshift=6cm]
\draw (2,0)--(2,6);
\draw (4,0)--(4,6);
\node[align=center] at (1,3) {\small $c_j=c_{j'}$};
\node[align=center] at (3,3) {\small $c_j \sneq c_{j'}$};
\node[align=center] at (5,3) {\small $c_j \nsim c_{j'}$};
\node[above right=-2pt] at (0,0) {\footnotesize \tt 25};
\node[above right=-2pt] at (2,0) {\footnotesize \tt 26};
\node[above right=-2pt] at (4,0) {\footnotesize \tt 27};
\filldraw[color=green,opacity=0.05](0,0) rectangle (6,6);
\end{scope}

\begin{scope}[xshift=6cm,yshift=6cm]
\draw (0,0)--(6,6);
\node[align=center] at (2,4) {\large $s_k=s_{k'}$};
\node[align=center] at (4,2) {\large $s_k \neq s_{k'}$};
\node[right=-2pt] at (0,1) {\footnotesize \tt 30};
\node[above=-2pt] at (1,0) {\footnotesize \tt 31};
\filldraw[color=blue,opacity=0.05](0,0) rectangle (6,6);
\end{scope}

\begin{scope}[xshift=12cm,yshift=6cm]
\draw (2,0)--(2,6);
\draw (4,0)--(4,6);
\draw (0,2)--(2,4);
\draw (2,2)--(4,4);
\draw (4,2)--(6,4);
\node[align=center] at (1,4.6) {\small $c_j=c_{j'}$\\ \small $s_k=s_{k'}$};
\node[align=center] at (3,4.6) {\small $c_j \sneq c_{j'}$\\ \small $s_k=s_{k'}$};
\node[align=center] at (5,4.6) {\small $c_j \nsim c_{j'}$\\ \small $s_k=s_{k'}$};
\node[align=center] at (1,1.3) {\small $c_j=c_{j'}$\\ \small $s_k \neq s_{k'}$};
\node[align=center] at (3,1.3) {\small $c_j \sneq c_{j'}$\\ \small $s_k \neq s_{k'}$};
\node[align=center] at (5,1.3) {\small $c_j \nsim c_{j'}$\\ \small $s_k \neq s_{k'}$};
\node at (0.5,3) {\footnotesize \tt 32};
\node at (1.5,3) {\footnotesize \tt 33};
\node at (2.5,3) {\footnotesize \tt 34};
\node at (3.5,3) {\footnotesize \tt 35};
\node at (4.5,3) {\footnotesize \tt 36};
\node at (5.5,3) {\footnotesize \tt 37};
\filldraw[color=cyan,opacity=0.05](0,0) rectangle (6,6);
\end{scope}

\begin{scope}[xshift=18cm,yshift=6cm]
\draw[decorate,decoration={coil,aspect=0,amplitude=1pt}] (3,0)--(3,6);
\draw (0,1.5)--(3,4.5);
\draw (3,1.5)--(6,4.5);
\node[align=center] at (1.4,4.6) {$\frown$\\$s_k=s_{k'}$};
\node[align=center] at (4.4,4.6) {$\not\frown$\\$s_k = s_{k'}$};
\node[align=center] at (1.6,1.4) {$\frown$\\$s_k \neq s_{k'}$};
\node[align=center] at (4.6,1.4) {$\not\frown$\\$s_k\neq s_{k'}$};
\node[above right=-2pt] at (0,2.2) {\footnotesize \tt 54};
\node[above right=-2pt] at (0,0) {\footnotesize \tt 55};
\node[above right=-2pt] at (3,2.2) {\footnotesize \tt 56};
\node[above right=-2pt] at (3,0) {\footnotesize \tt 57};
\filldraw[color=black!50!green,opacity=0.05](0,0) rectangle (6,6);
\end{scope}

\begin{scope}[xshift=0cm,yshift=0cm]
\draw[decorate,decoration={coil,aspect=0,amplitude=1pt}] (3,0)--(3,6);
\draw[decorate,decoration={coil,aspect=0,amplitude=1pt}] (0,3)--(6,3);
\node[align=center] at (1.5,4.5) {$b_\ell$\\ $\asymp$ \\$b_{\text{box}(i,j)}$};
\node[align=center] at (4.5,4.5) {$b_\ell$\\ $\smile$ \\$b_{\text{box}(i,j)}$};
\node[align=center] at (1.5,1.5) {$b_\ell$\\ $\frown$ \\$b_{\text{box}(i,j)}$};
\node[align=center] at (4.5,1.5) {$b_\ell$\\ $\not\asymp$ \\$b_{\text{box}(i,j)}$};
\node[above right=-2pt] at (0,3) {\footnotesize \tt 42};
\node[above right=-2pt] at (3,3) {\footnotesize \tt 43};
\node[above right=-2pt] at (0,0) {\footnotesize \tt 44};
\node[above right=-2pt] at (3,0) {\footnotesize \tt 45};
\filldraw[color=orange,opacity=0.05](0,0) rectangle (6,6);
\end{scope}

\begin{scope}[xshift=6cm,yshift=0cm]
\draw[decorate,decoration={coil,aspect=0,amplitude=1pt}] (0,3)--(6,3);
\draw (1.5,0)--(4.5,3);
\draw (1.5,3)--(4.5,6);
\node[align=center] at (1.5,4.7) {$\smile$\\$s_k=s_{k'}$};
\node[align=center] at (4.5,4.3) {$\smile$\\$s_k \neq s_{k'}$};
\node[align=center] at (1.5,1.7) {$\not\smile$\\$s_k=s_{k'}$};
\node[align=center] at (4.5,1.3) {$\not\smile$\\$s_k\neq s_{k'}$};
\node[above right=-2pt] at (0,3) {\footnotesize \tt 50};
\node[above right=-2pt] at (2,3) {\footnotesize \tt 51};
\node[above right=-2pt] at (0,0) {\footnotesize \tt 52};
\node[above right=-2pt] at (2,0) {\footnotesize \tt 53};
\filldraw[color=black!50!red,opacity=0.05](0,0) rectangle (6,6);
\end{scope}

\begin{scope}[xshift=12cm,yshift=0cm]
\draw[decorate,decoration={coil,aspect=0,amplitude=1pt}] (3,0)--(3,6);
\draw (0,1.5)--(3,4.5);
\draw (3,1.5)--(6,4.5);
\node[align=center] at (1.4,4.6) {$\frown$\\$s_k=s_{k'}$};
\node[align=center] at (4.4,4.6) {$\not\frown$\\$s_k = s_{k'}$};
\node[align=center] at (1.6,1.4) {$\frown$\\$s_k \neq s_{k'}$};
\node[align=center] at (4.6,1.4) {$\not\frown$\\$s_k\neq s_{k'}$};
\node[above right=-2pt] at (0,2.2) {\footnotesize \tt 58};
\node[above right=-2pt] at (0,0) {\footnotesize \tt 59};
\node[above right=-2pt] at (3,2.2) {\footnotesize \tt 60};
\node[above right=-2pt] at (3,0) {\footnotesize \tt 61};
\filldraw[color=black!50!green,opacity=0.05](0,0) rectangle (6,6);
\end{scope}

\begin{scope}[xshift=18cm,yshift=0cm]
\draw[decorate,decoration={coil,aspect=0,amplitude=1pt}] (3,0)--(3,6);
\draw[decorate,decoration={coil,aspect=0,amplitude=1pt}] (0,3)--(6,3);
\draw (0,0)--(3,3);
\draw (0,3)--(3,6);
\draw (3,0)--(6,3);
\draw (3,3)--(6,6);
\node[align=center] at (1,5) {$\asymp,=$};
\node[align=center] at (2,4) {$\asymp,\neq$};
\node[align=center] at (4,5) {$\smile,=$};
\node[align=center] at (5,4) {$\smile,\neq$};
\node[align=center] at (1,2) {$\frown,=$};
\node[align=center] at (2,1) {$\frown,\neq$};
\node[align=center] at (4,2) {$\not\asymp,=$};
\node[align=center] at (5,1) {$\not\asymp,\neq$};
\node[right=-2pt]  at (0,4) {\footnotesize \tt 62};
\node[above=-2pt]  at (1,3.1) {\footnotesize \tt 63};
\node[right=-2pt]  at (3,4) {\footnotesize \tt 64};
\node[above=-2pt]  at (4,3.1) {\footnotesize \tt 65};
\node[right=-2pt]  at (0,1) {\footnotesize \tt 66};
\node[above=-2pt]  at (1,0.1) {\footnotesize \tt 67};
\node[right=-2pt]  at (3,1) {\footnotesize \tt 68};
\node[above=-2pt]  at (4,0.1) {\footnotesize \tt 69};
\filldraw[color=gray,opacity=0.05](0,0) rectangle (6,6);
\end{scope}
\end{tikzpicture}
\caption{Relations on edges of $G_{hw}$}
\label{fig:edge-rels}
\end{figure}

\begin{prop}
The relations $R_1,\dots,R_{69}$ given in Figure~\ref{fig:edge-rels} define a coherent configuration on on $E(G_{hw})$.
\end{prop}

\begin{proof}
The relations $=,\sneq,\not\sim$ are orbits of $\Sym_h \wr \Sym_w$ (respectively 
$\Sym_w \wr \Sym_r$) acting on pairs of rows (columns).  The direct product of second factors $\Sym_w \times \Sym_h$ induces relations $=,\smile,\frown,\asymp$ on ordered pairs of boxes. The relations $\smile,\not\smile$ (respectivley $\frown,\not\frown$) also define orbits on ordered pairs of rows (columns) with boxes.  
Finally, the relations $=,\neq$ define orbits of $\Sym_n$ acting on pairs of symbols.  
It follows that the set of relations $\{R_1,\dots,R_{69}\}$ matches the coherent configuration from the action of $\Gamma_{hw}$ on $E(G_{hw})$.
\end{proof}

The diagonal relation $\{(x,x): x \in X\}$ is a union of $R_1,R_{16},R_{32},R_{62}$.  For any relation $R_i$, its transpose $R_i^\top$ can be identified directly from Figure~\ref{fig:edge-rels}.
With extensive case analysis, it would be technically possible to demonstrate formulas for the structure constants $p_{ij}^k$.  However, to avoid presenting such details and to reduce errors, we implemented the following computer-assisted procedure:
\vspace{-11pt}
\begin{itemize}
\item
first, we argue that $p_{ij}^k$ belongs to $\Z[h,w]$, and is at most quadratic in each of $h,w$;
\item
next, we compute all structure constants explicitly for the nine cases $2 \le h,w \le 4$;
\item
finally, we interpolate this data to arrive at symbolic expressions for $p_{ij}^k$.
\end{itemize}

We discuss these points in a little more detail.

\begin{prop}
\label{prop:struct-linear}
In the coherent configuration defined by $\Gamma_{hw}$, each structure constant $p_{ij}^k$ is a polynomial of degree at most $2$ in each of $h$ and $w$.
\end{prop}

\begin{proof}
Fix two edges $x,z \in E(G_{hw})$ with $(x,z) \in R_k$.  The quantity $p_{ij}^k$ counts the edges $y \in E(G_{hw})$ with $(x,y) \in R_i$ and $(y,z) \in R_j$.  This quantity is zero unless the indices $i$ and $j$ simultaneously allow one of the four types of edges for $y$.  Given indices $i$ and $j$ which admit a choice of $y$, we must choose either a row-column pair, a row-symbol pair, a column-symbol pair, or a box-symbol pair.  The two components of each pair can be selected separately, leading to a product of choices for the two components.  Each factor is easily seen to have degree at most one in both $h$ and $w$.  The number of choices for a row is an element of $\{0,1,h-2,h-1,h,n-2h,n-h,n\}$.
Similarly, the number of choices for a column is an element of $\{0,1,w-2,w-1,w,n-2w,n-w,n\}$.  The number of choices for a symbol is an element of $\{0,1,n-2,n-1,n\}$. Finally, the number of choices for a box is a product of an element of $\{0,1,h-2,h-1,h\}$ with an element of $\{0,1,w-2,w-1,w\}$.
\end{proof}

The choice of nine cases $2 \le h,w \le 4$ suffices because of the degree bound in Proposition~\ref{prop:struct-linear}.  The computation was carried out on computer by explicitly listing all $4h^2w^2$ edges and counting incidences.  This took several minutes for the larger cases.  From this, the interpolation in (3) can be performed using a $9 \times 9$ Vandermonde matrix based on the terms $1,h,w,h^2,hw,w^2,h^2w,hw^2,h^2w^2$, where $(h,w) \in \{2,3,4\}^2$.

For each relation index $i=1,\dots,69$, we consider its corresponding adjacency matrix $A_i$.  Let $\mathfrak{A}$ denote the $\R$-vector space spanned by the $A_i$.  Since the relations form a coherent configuration, $\mathfrak{A}$ is closed under matrix multiplication, and hence forms an algebra.

If we view each relation as a graph, then $\{R_i:i=1,\dots,69\}$ is a decomposition of the line graph of $G_{hw}$ into regular graphs.  The degrees of these graphs are the nonzero row sums of the corresponding adjacency matrices.  We give the degrees $d_i$ for each of the relations in Table~\ref{tab:degrees}.  These are arranged into columns according to the four edge types: row-column, row-symbol, column-symbol, and symbol-box.
Consider, for example, the degree $d_{47}$.  Given a row-symbol edge, say $r_1s_1$, this degree counts the symbol-box edges $s_k b_\ell$ with $s_k \neq s_1$ and $b_\ell$ in the same row bundle as $r_1$.  There are $n-1$ choices for $s_k$ and $h$ choices for $b_\ell$, since every row is incident with exactly $n/w=h$ boxes.  So $d_{47}=(n-1)h$.

\begin{table}[htbp]
\small
$$\begin{array}{|rr|rr|rr|rr|}
\hline
i &  d_i & i &  d_i  & i & d_i & i &  d_i \\
\hline
1	&	1	&	13	&	n	&	25	&	n	&	42	&	n	\\
2	&	w-1	&	14	&	n(h-1)	&	26	&	n(w-1)	&	43	&	n(h-1)	\\
3	&	(h-1)w	&	15	&	nh(w-1)	&	27	&	n(h-1)w	&	44	&	n(w-1)	\\
4	&	h-1	&	16	&	1	&	30	&	n	&	45	&	n(h-1)(w-1)	\\
5	&	(h-1)(w-1)	&	17	&	n-1	&	31	&	n(n-1)	&	50	&	h	\\
6	&	(h-1)^2w	&	18	&	h-1	&	32	&	1	&	51	&	(n-1)h	\\
7	&	h(w-1)	&	19	&	(n-1)(h-1)	&	33	&	n-1	&	52	&	h(w-1)	\\
8	&	h(w-1)^2	&	20	&	h(w-1)	&	34	&	w-1	&	53	&	(n-1)h(w-1)	\\
9	&	n(h-1)(w-1)	&	21	&	(n-1)h(w-1)	&	35	&	(n-1)(w-1)	&	58	&	w	\\
10	&	n	&	28	&	n	&	36	&	(h-1)w	&	59	&	(n-1)w	\\
11	&	n(h-1)	&	29	&	n(n-1)	&	37	&	(n-1)(h-1)w	&	60	&	(h-1)w	\\
12	&	nh(w-1)	&	46	&	h	&	54	&	w	&	61	&	(n-1)(h-1)w	\\
22	&	n	&	47	&	(n-1)h	&	55	&	(n-1)w	&	62	&	1	\\
23	&	n(w-1)	&	48	&	h(w-1)	&	56	&	(h-1)w	&	63	&	n-1	\\
24	&	n(h-1)w	&	49	&	(n-1)h(w-1)	&	57	&	(n-1)(h-1)w	&	64	&	h-1	\\
38	&	n	&		&		&		&		&	65	&	(n-1)(h-1)	\\
39	&	n(h-1)	&		&		&		&		&	66	&	w-1	\\
40	&	n(w-1)	&		&		&		&		&	67	&	(n-1)(w-1)	\\
41	&	n(h-1)(w-1)	&		&		&		&		&	68	&	(h-1)(w-1)	\\
	&		&		&		&		&		&	69	&	(n-1)(h-1)(w-1)	\\
\hline
\end{array}$$
\normalsize
\caption{Relation degrees $d_i$; alternatively the nonzero row sums of $A_i$}
\label{tab:degrees}
\end{table}

\subsection{The coefficient matrix $M$}

Recall that $W$ denotes the $\{0,1\}$ inclusion matrix of edges versus tiles in $G_{hw}$, and that $M=WW^\top$.  We computed the rank and a basis for the kernel of $M$ in Section~\ref{sec:setup}.  A key next observation is that $M$ belongs to our adjacency algebra.

\begin{prop}
The matrix $M=WW^\top$ lies in $\mathfrak{A}$, with
\begin{align}
\label{eq:M}
M=&hw(A_1+A_{16}+A_{32}+A_{62})
+w(A_{46}+A_{50})
+h(A_{54}+A_{58})\\
\nonumber
&+A_{10}+A_{13}+A_{22}+A_{25}+A_{28}+A_{30}+A_{38}+A_{42}.
\end{align}
\end{prop}

\begin{proof}
Given two edges $e,f$ in $G_{hw}$, the entry $M(e,f)$ equals the number of tiles containing both $e$ and $f$.  Since a tile contains exactly one row, column, box, and symbol, this number is zero whenever $e \cup f$ contains two distinct vertices of the same type.  Moreover, since the box in a tile must correspond with the row-column pair, $M(e,f)$ is zero if $e \cup f \supset \{r_i,b_\ell\}$ or $\{c_j,b_\ell\}$  with, respectively $r_i \not\smile b_\ell$ or $c_j \not\frown b_\ell$. It suffices to consider those remaining cases when there exist tiles extending $e \cup f$.

If $e=f$, we claim that there are $n$ such tiles, regardless of the type of edge.  For $e=\{r_i,c_j\}$, any of the $n$ symbols extend $e$ to a tile (and there is a unique box involved).  For $e=\{r_i,s_k\}$, there are any of $n$ columns (with one corresponding box for each) extending $e$.  This is similar when we exchange the roles of rows and columns.  Finally, for $e=\{b_\ell,s_k\}$, any of the $hw=n$ cells $(i,j)$ with $\mathrm{box}(i,j)=\ell$ extend $e$ to a tile.  The identity relation in our setup decomposes into the identity on the four types of edges; in terms of matrices,
$$I=A_1+A_{16}+A_{32}+A_{62}.$$
We have shown that the diagonal entries of $M$, and hence the coefficient for each of these four adjacency matrices, equals $n$.

Next, consider $e=\{r_i,s_k\}$ and $f=\{b_\ell,s_k\}$.  In the event that $r_i \smile b_\ell$, we obtain $w$ possible columns $c_j$ such that $\mathrm{box}(i,j)=\ell$, and $e \cup f \cup \{c_j\}$ defines a valid tile.  The two relations corresponding to this choice of $e$ and $f$ (transposes of each other) have indices $46$ and $50$ in our labeling.  If instead we take $e=\{c_j,s_k\}$ for $c_j \frown b_\ell$, there are likewise exactly $h$ extensions to a tile via some row $r_i$.  This choice corresponds to relations numbered $54$ and $58$.

Finally, in each of the following possibilities for $\{e,f\}$, there is a unique tile $r_ic_js_k b_\ell$ extending $e \cup f$, where $\ell=\mathrm{box}(i,j)$:
$$\{r_i c_j, r_i s_k\},\{r_i c_j, c_j s_k\}, \{r_i s_k, c_j s_k\}, \{r_i c_j, b_\ell s_k\}.$$
The corresponding relation labels are $10,13,22,25,28,30,38,42$.
\end{proof}

The structure of entries of $M$ is depicted in Figure~\ref{fig:M}.  On the left, we present $M$ as a block matrix, whose block partition corresponds to the four edge types.  Each block is an $n^2 \times n^2$ matrix which can be factored as a Kronecker product.  It is convenient to slightly abuse the Kronecker product in the following way.  In forming $A \otimes B$, each factor will be indexed by one of our four Sudoku objects: rows, columns, symbols, and boxes.  The product is then indexed by corresponding pairs of elements.  For instance, the $(1,2)$-block of $M$ can be represented as $I_r \otimes J_{cs}$, where $I_r$ is the identity matrix indexed by $\{r_1,\dots,r_n\}$ and $J_{cs}$ is the all-ones matrix whose rows are indexed by $\{c_1,\dots,c_n\}$ and columns are indexed by $\{s_1,\dots,s_n\}$.  The latter can be factored as $\mathbf{j}_c \otimes \mathbf{j}_s^\top$, where $\mathbf{j}$ is an $n \times 1$ all-ones vector and the subscript indicates the indexing set.  The $(1,2)$-block of $M$ has rows indexed by edges $e=r_i c_j$, columns indexed by edges $f=r_{i'} s_{k'}$, and the $(e,f)$-entry is $1$ if and only if $i=i'$.  This exactly recovers the condition for $e$ and $f$ sharing a common tile.  Other blocks of $M$ are similar.  We use $H_{rb}$ to denote the zero-one matrix indexed by rows versus boxes in which $H_{rb}(r_i,b_\ell) = 1$ if and only if $r_i \smile b_\ell$.  We use $H_{cb}$ analogously for columns.  Finally, $H_{rcb}$ is $n^2 \times n$, indexed by row-column edges versus boxes, and $H_{rcb}(r_ic_j,b_\ell)=1$ if and only if box$(i,j)=\ell$.

\begin{figure}[htbp]
\begin{center}
\begin{tikzpicture}
\draw (-3,-3) rectangle (3,3);
\draw[dotted] (-3,-1.5)--(3,-1.5);
\draw[dotted] (-3,0)--(3,-0);
\draw[dotted] (-3,1.5)--(3,1.5);
\draw[dotted] (-1.5,-3)--(-1.5,3);
\draw[dotted] (0,-3)--(0,3);
\draw[dotted] (1.5,-3)--(1.5,3);
\node at (-2.25,2.25) {$nI_r \otimes I_c$};
\node [rotate=315] at (-.75,2.25) {\small $I_r \otimes \mathbf{j}_c \otimes \mathbf{j}_s^\top$};
\node [rotate=315] at (.75,2.25) {\small $\mathbf{j}_r \otimes I_c \otimes \mathbf{j}_s^\top$};
\node at (2.25,2.25) {\small $H_{rcb} \otimes \mathbf{j}_s^\top$};
\node [rotate=315] at (-2.25,.75) {\small $I_r \otimes \mathbf{j}_c^\top \otimes \mathbf{j}_s$};
\node at (-.75,.75) {$nI_r \otimes I_s$};
\node [rotate=315] at (.75,.75) {\small $\mathbf{j}_r \otimes \mathbf{j}_c^\top \otimes I_s$};
\node at (2.25,.75) {\small $wH_{rb} \otimes I_s$};
\node [rotate=315] at (-2.25,-.75) {\small $\mathbf{j}_r^\top \otimes I_c \otimes \mathbf{j}_s$};
\node [rotate=315] at (-.75,-.75) {\small $\mathbf{j}_r^\top \otimes \mathbf{j}_c \otimes I_s$};
\node at (.75,-.75) {$nI_c \otimes I_s$};
\node at (2.25,-.75) {\small $hH_{cb} \otimes I_s$};
\node at (-2.25,-2.25) {\small $H_{rcb}^\top \otimes \mathbf{j}_s$};
\node at (-.75,-2.25) {\small $wH_{rb}^\top \otimes I_s$};
\node at (.75,-2.25) {\small $hH_{cb}^\top \otimes I_s$};
\node at (2.25,-2.25) {$nI_b \otimes I_s$};
\node at (3.5,0) {$=$};
\node at (7,0) {\includegraphics[scale=0.25]{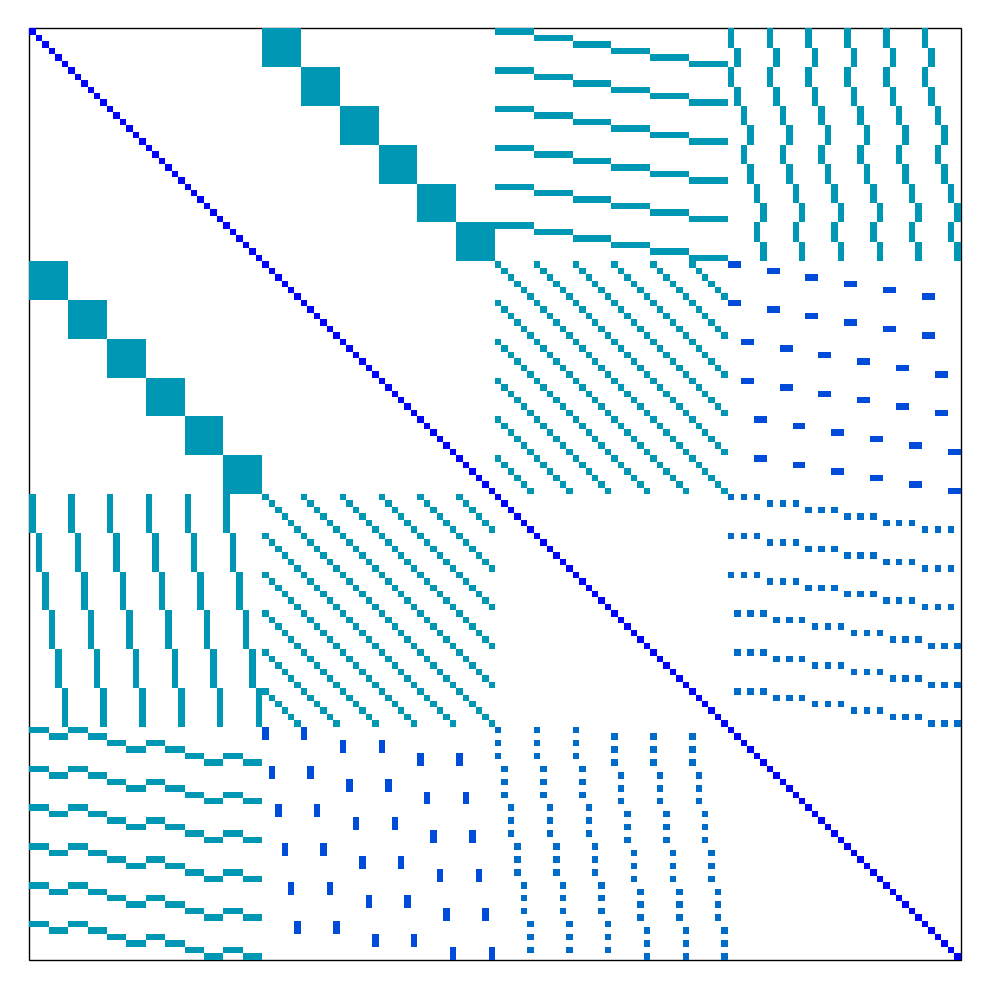}};
\end{tikzpicture}
\caption{Illustration of the block matrix structure of $M$}
\label{fig:M}
\end{center}
\end{figure}

On the right of Figure~\ref{fig:M}, we display the locations of nonzero entries as a graphic, illustrated in the case $h=2$, $w=3$.  The diagonal has entries $n=hw$.  Blocks $(2,4)$ and $(4,2)$ correspond to $A_{46}$ and $A_{50}$, with coefficient $w$.  Blocks $(3,4)$ and $(4,3)$ correspond to $A_{54}$ and $A_{58}$, with coefficient $h$.  The other blocks correspond to the remaining terms with coefficient 1.

\section{Spectral decomposition of $M$}
\label{sec:spectral}

\subsection{Eigenvalues and eigenvectors}
\label{sec:eigenvalues}

Since $M=WW^\top$, we know it is symmetric and hence has real eigenvalues. We also have $\mathrm{rank}(M)=\mathrm{rank}(W)$, so from Section~\ref{sec:kernel}, we know that zero is an eigenvalue of $M$ with multiplicity $3n + (h + w)(n - 1)$.  Moreover, $M$ has constant row sums equal to $4n$, since every edge belongs to $n$ tiles, and every tile has four edges.  This gives an eigenvalue $4n$ corresponding to the one-dimensional eigenspace of constant vectors.  

In this section, we compute all other eigenvalues and corresponding eigenvectors for $M$.  
By \eqref{eq:M}, we know that $M \in \mathfrak{A}$.
Later, a generalized inverse for $M$ is expressed with a list of coefficients in $\mathfrak{A}$.  For the discovery of these coefficients, it is helpful to have a good understanding of the spectral decomposition of $M$. This is summarized here, with more details and verifications for eigenvectors appearing in the remainder of this subsection.

\begin{prop}
\label{prop:Mspec}
The eigenvalues of $M$ are $\theta_j=jn$, $j=0,1,\dots,4$.  Each eigenspace has a basis of eigenvectors consisting of vectors with entries in $\{0,\pm 1\}$.
\end{prop}

We have discussed $\theta_0=0$ and $\theta_4=4n$ earlier, so we turn our focus to $\theta_1,\theta_2,\theta_3$.  Below, we describe different varieties of eigenvectors (A), (B), etc., for each of these $\theta_j$.  A basis for each eigenspace can be found by taking a union of linearly independent vectors over the different varieties.  Making a selection of linearly independent vectors of the indicated size within each variety can be done using relations as in Section~\ref{sec:kernel}.  More details can be found in \cite{Kate}.

We give an informal description and brief verification for each eigenvector. Checking that $M \vv = \theta_j \vv$ can be done as follows.  Take each edge $f \in E(G_{hw})$ and extend to a tile $t \supset f$ in all possible ways.  Then, sum the values of $\vv$ on the four edges of $t$, and check that this total equals $\theta_j \vv(f)$.  This often equals zero, either from canceling signs or when the support of $\vv$ is disjoint from the relevant tiles $t$.  Figures~\ref{fig:evecs1}, \ref{fig:evecs2} and \ref{fig:evecs3} give diagrams illustrating the eigenvector varieties in the case $(h,w)=(2,3)$.  In these diagrams, the four sections correspond to the four edge types: row-column (top left), row-symbol (top right), symbol-column (bottom left), and box-symbol (bottom right).  Symbols $+$ and $-$ denote vector entries $1$ and $-1$, respectively, and blanks represent $0$ in the corresponding positions.  

$\bullet$ $\theta_1=n$; eigenspace dimension $4n^2-(2n-3)(h+w)-5n-1$

\begin{figure}[htbp]
\begin{center}
\begin{tikzpicture}
\node at (-0.6,1.2) {(A)};
\rcscope{0/0,3/1}{3/0,0/1};
\rsscope{}{};
\csscope{}{};
\bsscope{}{};
\end{tikzpicture}
\hspace{5mm}
\begin{tikzpicture}
\node at (-0.6,1.2) {(B)};
\rcscope{}{};
\rsscope{0/0,1/1}{0/1,1/0};
\csscope{}{};
\bsscope{}{};
\end{tikzpicture}
\hspace{5mm}
\begin{tikzpicture}
\node at (-0.6,1.2) {(C)};
\rcscope{}{};
\rsscope{}{};
\csscope{}{};
\bsscope{0/0,1/1,1/2,0/3}{1/0,0/1,0/2,1/3};
\end{tikzpicture}
\caption{Eigenvectors for $\theta_1=n$}
\label{fig:evecs1}
\end{center}
\end{figure}
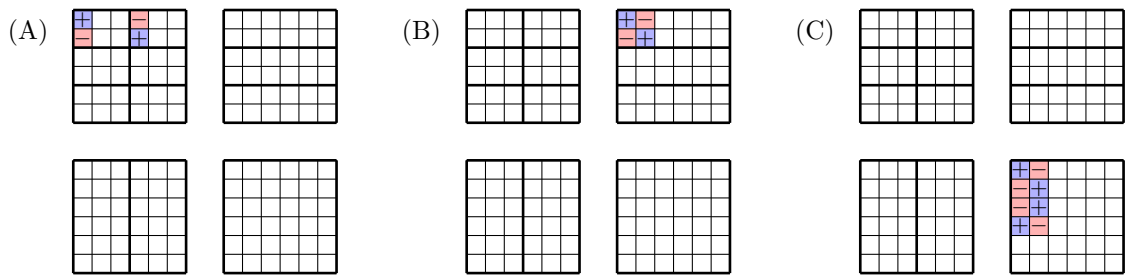

(A) Opposite signs on two distinct rows and two distinct columns, at least one pair of which is in a common bundle.  There are $(n-1)^2-(h-1)(w-1)$ linearly independent vectors of this kind. 

If $r_{i_0},r_{i_1}$ are the rows and $c_{j_0},c_{j_1}$ are the columns, then the entries are given explicitly by $\vv(r_ic_j) = (-1)^{\alpha+\beta}$ if $(i,j)=(i_\alpha,j_\beta)$, and otherwise $\vv(e)=0$.
For an edge $f=r_ic_j$, we have $n$ tiles extending $f$ corresponding to a choice of symbol $s_k$.  Each such tile contains at most one non-vanishing edge, namely that corresponding to $f$.  So $M\vv(f)=n\vv(f)$.
For $f$ of type row-symbol, column-symbol or box-symbol, we have $M\vv(f)=0$, either from cancellation or disjoint supports.  Importantly, having either $r_{i_0} \sim r_{i_1}$ or $c_{j_0} \sim c_{j_1}$ ensures cancellation within each box.

(B) Opposite signs on two distinct rows (or columns) in the same bundle and on two distinct symbols.  There are $(n-1)(h(w-1)+w(h-1))$ linearly independent vectors of this kind. 

If $r_{i_0},r_{i_1}$ are the rows and $s_{k_0},s_{k_1}$ are the symbols, then the entries are given explicitly by $\vv(r_is_k) = (-1)^{\alpha+\gamma}$ if $(i,k)=(i_\alpha,k_\gamma)$, and otherwise $\vv(e)=0$.  Verification that $M\vv=n\vv$ is similar to (A).

(C) Alternating signs on a rectangle of boxes and opposite signs on two distinct symbols.  There are $(n-1)(h-1)(w-1)$ linearly independent vectors of this kind.

Suppose $\ell_{\alpha\beta}$ are the four box indices, where $\alpha,\beta \in \{0,1\}$ tell us the chosen row/column bundles, respectively.  As in (B), let $k_\gamma$ be the chosen symbol indices, $\gamma \in \{0,1\}$.  The entries of the eigenvector are given by $\vv(s_k b_\ell) = (-1)^{\alpha+\beta+\gamma}$ when $(k,\ell) = (k_\gamma,\ell_{\alpha\beta})$, and otherwise $\vv(e)=0$.  Cancellation occurs if we sum over rows, columns, or symbols.  For a symbol-box edge $f=s_k b_\ell$, the $n$ tiles extending $f$ correspond to a choice of entry in box $\ell$. This picks up the value of $\vv(f)$ with a multiplicity of $n$. So $M\vv=n \vv$.

$\bullet$  $\theta_2=2n$; eigenspace dimension $(2n-h-w)+(n-1)(h+w-2)+(h-1)(w-1)=(n-3)(h+w-1)+2n$.

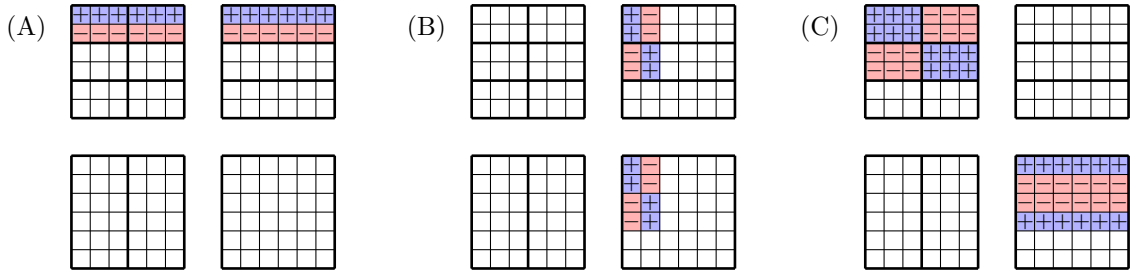
\begin{figure}[htbp]
\begin{center}
\begin{tikzpicture}
\node at (-0.6,1.2) {(A)};
\rcscope{0/0,1/0,2/0,3/0,4/0,5/0}{0/1,1/1,2/1,3/1,4/1,5/1};
\rsscope{0/0,1/0,2/0,3/0,4/0,5/0}{0/1,1/1,2/1,3/1,4/1,5/1};
\csscope{}{};
\bsscope{}{};
\end{tikzpicture}
\hspace{5mm}
\begin{tikzpicture}
\node at (-0.6,1.2) {(B)};
\rcscope{}{};
\rsscope{0/0,0/1,1/2,1/3}{1/0,1/1,0/2,0/3};
\csscope{}{};
\bsscope{0/0,0/1,1/2,1/3}{1/0,1/1,0/2,0/3};
\end{tikzpicture}
\hspace{5mm}
\begin{tikzpicture}
\node at (-0.6,1.2) {(C)};
\rcscope{0/0,1/0,2/0,0/1,1/1,2/1,3/2,4/2,5/2,3/3,4/3,5/3}{0/2,1/2,2/2,0/3,1/3,2/3,3/0,4/0,5/0,3/1,4/1,5/1};
\rsscope{}{};
\csscope{}{};
\bsscope{0/0,1/0,2/0,3/0,4/0,5/0,0/3,1/3,2/3,3/3,4/3,5/3}
{0/1,1/1,2/1,3/1,4/1,5/1,0/2,1/2,2/2,3/2,4/2,5/2};
\end{tikzpicture}
\caption{Eigenvectors for $\theta_2=2n$}
\label{fig:evecs2}
\end{center}
\end{figure}

(A) Opposite signs on two distinct rows in the same bundle; constant on all columns and symbols.
A similar variety exists with rows and columns swapped.  There are $h(w-1)+w(h-1)=2n-h-w$ linearly independent vectors of this kind. 

If $r_{i_0},r_{i_1}$ are the rows, then the eigenvector entries are $\vv(r_i c_j) = \vv(r_is_k) = (-1)^{\alpha}$ when $i=i_\alpha$, and otherwise $\vv(e)=0$ for all other edges.  An edge $f=r_i c_j$ or $r_i s_k$ has exactly $n$ extensions to a tile, each of which has two edges of a common sign.  So $M\vv(f)=2n \vv(f)$ in those cases.  It is easy to see that $M\vv(f)=0$ on all other edges due to cancellation on rows.

(B) Opposite signs on both rows and boxes of two distinct row bundles; opposite signs on symbols.  A similar variety exists with rows and columns swapped. 
There are $(n-1)(h+w-2)$ linearly independent vectors of this kind.

If $f$ is a row-column edge, the cancellation on symbols gives $M\vv(f)=0$.  Likewise, if $f$ is a column-symbol edge, the cancellation on rows gives $M\vv(f)=0$.  For $f$ of either of the other two edge types, there are $n$ extensions to a tile, and again the nonzero edges (if any) agree in sign.

(C) Alternating signs on a rectangle of boxes; constant on all symbols and on entries within each box.
There are $(h-1)(w-1)=n-h-w+1$ linearly independent vectors of this kind.

For row-symbol or column-symbol edges, the extension to a tile leads to cancellation.  For a row-column edge $f$, there
$n$ extensions to a tile by selecting a symbol, and each has two matching edges from the entry and box.  So $M\vv(f)=2n\vv(f)$.  Similarly, for a box-symbol edge $f$, we have $M\vv(f)=2n\vv(f)$.

$\bullet$ $\theta_3=3n$; eigenspace dimension $n+h+w-3$

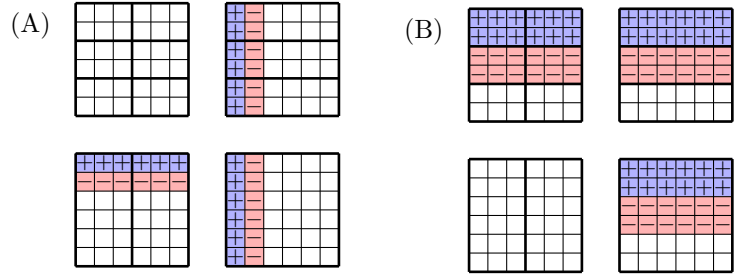
\begin{figure}[htbp]
\begin{center}
\begin{tikzpicture}
\node at (-0.6,1.2) {(A)};
\rcscope{}{};
\rsscope{0/0,0/1,0/2,0/3,0/4,0/5}{1/0,1/1,1/2,1/3,1/4,1/5};
\csscope{0/0,1/0,2/0,3/0,4/0,5/0}{0/1,1/1,2/1,3/1,4/1,5/1};
\bsscope{0/0,0/1,0/2,0/3,0/4,0/5}{1/0,1/1,1/2,1/3,1/4,1/5};
\end{tikzpicture}
\hspace{5mm}
\begin{tikzpicture}
\node at (-0.6,1.2) {(B)};
\begin{scope}[scale=0.25]
\foreach \a in {0,1,...,5}
 \foreach \b in {0,1}
  \posbox{\a}{\b};
\foreach \a in {0,1,...,5}
 \foreach \b in {2,3}
  \negbox{\a}{\b};
\framesix
\draw[line width=1pt] (3,0)--(3,6);
\draw[line width=1pt] (0,2)--(6,2);
\draw[line width=1pt] (0,4)--(6,4);
\end{scope}
\begin{scope}[xshift=2cm,scale=0.25]
\foreach \a in {0,1,...,5}
 \foreach \b in {0,1}
  \posbox{\a}{\b};
\foreach \a in {0,1,...,5}
 \foreach \b in {2,3}
  \negbox{\a}{\b};
\framesix
\draw[line width=1pt] (0,2)--(6,2);
\draw[line width=1pt] (0,4)--(6,4);
\end{scope}
\begin{scope}[yshift=-2cm,scale=0.25]
\framesix
\draw[line width=1pt] (3,0)--(3,6);
\end{scope}
\begin{scope}[xshift=2cm,yshift=-2cm,scale=0.25]
\foreach \a in {0,1,...,5}
 \foreach \b in {0,1}
  \posbox{\a}{\b};
\foreach \a in {0,1,...,5}
 \foreach \b in {2,3}
  \negbox{\a}{\b};
\framesix
\end{scope}
\end{tikzpicture}
\caption{Eigenvectors for $\theta_3=3n$}
\label{fig:evecs3}
\end{center}
\end{figure}

(A) Opposite signs on two distinct symbols; constant on all rows, columns, and boxes.
There are $n-1$ linearly independent vectors of this kind. 

If $k_0,k_1$ are the two symbols, then the eigenvector entries are $\vv(r_i s_k) = \vv(c_js_k) = \vv(b_\ell s_k) = (-1)^{\gamma}$ when $k=k_\gamma$, and otherwise $\vv(e)=0$ for all other edges.  If $f$ is any edge involving a symbol $s_{k_\gamma}$, the $n$ tiles extending $f$ each have (if any) three nonzero edges of matching sign.  So $M\vv(f)=3n\vv(f)$.  In other cases, it is easy to see that $M\vv(f)=0$ by cancellation.

(B) Opposite signs on both rows and boxes of two distinct row bundles; constant on all columns and symbols.  A similar variety exists with rows and columns swapped.  There are $(h-1)+(w-1)$ linearly independent vectors of this kind. 

The verification here is similar to (A), except that row bundles take the role of symbols.

\subsection*{Kronecker product}

As with our matrix $M$, it is possible to write the eigenvectors, including 
the kernel vectors from Section~\ref{sec:kernel}, using $\otimes$. 
Before doing so, we set up some ingredient vectors and conventions of notation.  Let $\vj_{r}$ denote the $n \times 1$ all-ones vector indexed by rows, and $\vj_c$, $\vj_s$, $\vj_b$ similar for columns, symbols and boxes, respectively.  Let $\ve_{r_i}$ denote the $n \times 1$ indicator vector of row $i$.  We omit the second subscript and write $\ve_r$ if $i$ is unimportant or clear from context.  Any difference $\ve_{r_i}-\ve_{r_{i'}}$ where $i \neq i'$ is denoted $\vd_r$.  Similar vectors are defined for columns, symbols, and boxes, with letters $c,s,b$ used accordingly.

We also need a few minor variants of these.  A difference $\ve_{r_i}-\ve_{r_{i'}}$ where $r_i \sneq  r_{i'}$ is denoted $\vd_r^\sim$; an analogous vector $\vd_c^\sim$ is defined for columns in the same bundle.  If $\ell_{00},\ell_{01},\ell_{10},\ell_{11}$ are box indices forming a rectangle, the alternating sum of their indicator vectors is denoted $\vd_b^\Box$.  Given two distinct row bundles, say $B_0$ and $B_1$, a combination $\sum_{r \in B_0} \ve_r - \sum_{r \in B_1}  \ve_r$ is denoted 
$\vd_r^\dagger$.  Similar vectors can be defined for columns and boxes.  We remark that all vectors involving the letter `$\vd$' have a sum of coordinates equal to zero.

As an important convention, when we take a Kronecker product of vectors from the above, it is regarded as a $4n^2 \times 1$ vector with zeros inserted as needed.  For instance, $\vj_r \otimes \vj_c$ has entries equal to $1$ supported on all row-column edges, and is assumed to vanish on edges of all other types (even though these do not appear in the product).  Additionally, choices are assumed to be canonical and consistent within a vector.  For instance, if $\vd_s$ appears twice, it is assumed to represent the same symbol indices; if $\vd_r^\dagger$ and $\vd_b^\dagger$ appear together, it is assumed the row bundle and box bundle coincide.  

\begin{table}[htbp]
\begin{tabular}{|c|c|r|l}
\cline{1-3}
eigenvalue & variety & eigenvector \\
\cline{1-3}
$0$ & (A) &  $\ve_r \otimes \vj_c - \ve_r \otimes \vj_s$ & $*$\\
 & (B) & $H_{rcb} \ve_b - \vj \otimes \ve_b$ \\
 & (C) & $\vd_r^\dagger \otimes \ve_s - \vd_b^\dagger \otimes \ve_s$  & $*$\\
\cline{1-3}
$n$ & (A) & 
$\vd_r^\sim \otimes \vd_c$ &  $*$\\
 & (B) & $\vd_r^\sim \otimes \vd_s$ &  $*$\\
 & (C) & $\vd_s \otimes \vd_b^\Box$\\
\cline{1-3}
$2n$ & (A) & $\vd_r^\sim \otimes \vj_c + \vd_r^\sim \otimes \vj_s$ & $*$\\
 & (B) & $\vd_r^\dagger \otimes \vd_s +\vd_b^\dagger \otimes \vd_s$ & $*$\\
 & (C) & $\vj_s \otimes \vd_b^\Box + H_{rcb} \vd_b^\Box$\\
\cline{1-3}
$3n$ & (A) & $\vd_s \otimes \vj_r + \vd_s \otimes \vj_c + \vd_s \otimes \vj_b$\\
 & (B) & $\vd_r^\dagger \otimes \vj_c + \vd_r^\dagger \otimes \vj_s + \vd_b^\dagger \otimes \vj_s$ &  $*$\\
\cline{1-3}
$4n$ &  & $\vj_r \otimes \vj_c + \vj_r \otimes \vj_s + \vj_c \otimes \vj_s + \vj_s \otimes \vj_b$ \\
\cline{1-3}
\end{tabular}

\caption{Eigenvectors of $M$; $*$ means minor variants exist}
\label{tab:evecs-Kronecker}
\end{table}
Although it does not convey much combinatorial understanding, eigenvectors in this form can in principle be checked using block matrix multiplication.  As an example, consider $\vv=\vd_r^\sim \otimes \vd_c$, a type (A) eigenvector for $\theta_1=n$.
The four blocks of the product $M\vv$ can be computed one at a time:
\begin{align*}
(nI_r \otimes I_c)(\vd_r^\sim \otimes \vd_c) &= n \vd_r^\sim \otimes \vd_c\\
(I_r \otimes J_{sc})(\vd_r^\sim \otimes \vd_c) &= \vd_r^\sim \otimes \mathbf{0} = \mathbf{0}\\
(J_{sr} \otimes I_c)(\vd_r^\sim \otimes \vd_c) &= \mathbf{0} \otimes \vd_c^\sim = \mathbf{0}\\
(H_{rcb}^\top \otimes \vj_{s})(\vd_r^\sim \otimes \vd_c) &= H_{rcb}^\top (\vd_r^\sim \otimes \vd_c) \otimes \vj_s = \mathbf{0} \otimes \vj_s = \mathbf{0},
\end{align*}
where the latter is because every box intersects both or neither of the row indices defining $\vd_r^\sim$.  It follows that $M\vv=n \vv$.

We next consider in more detail the projectors onto the eigenspaces of $M$. 

\subsection{Projectors and a generalized inverse for $M$}
\label{sec:projectors}

Since $M$ is symmetric, the projectors $E_j$ onto eigenspaces for $\theta_j$ are pairwise orthogonal idempotents summing to $I$.
Moreover, we have $E_j \in \mathfrak{A}$ for each $j$ as a general fact of coherent configurations; see for instance \cite{Higman}.  

\begin{figure}[htbp]
\begin{center}
\includegraphics[scale=0.25]{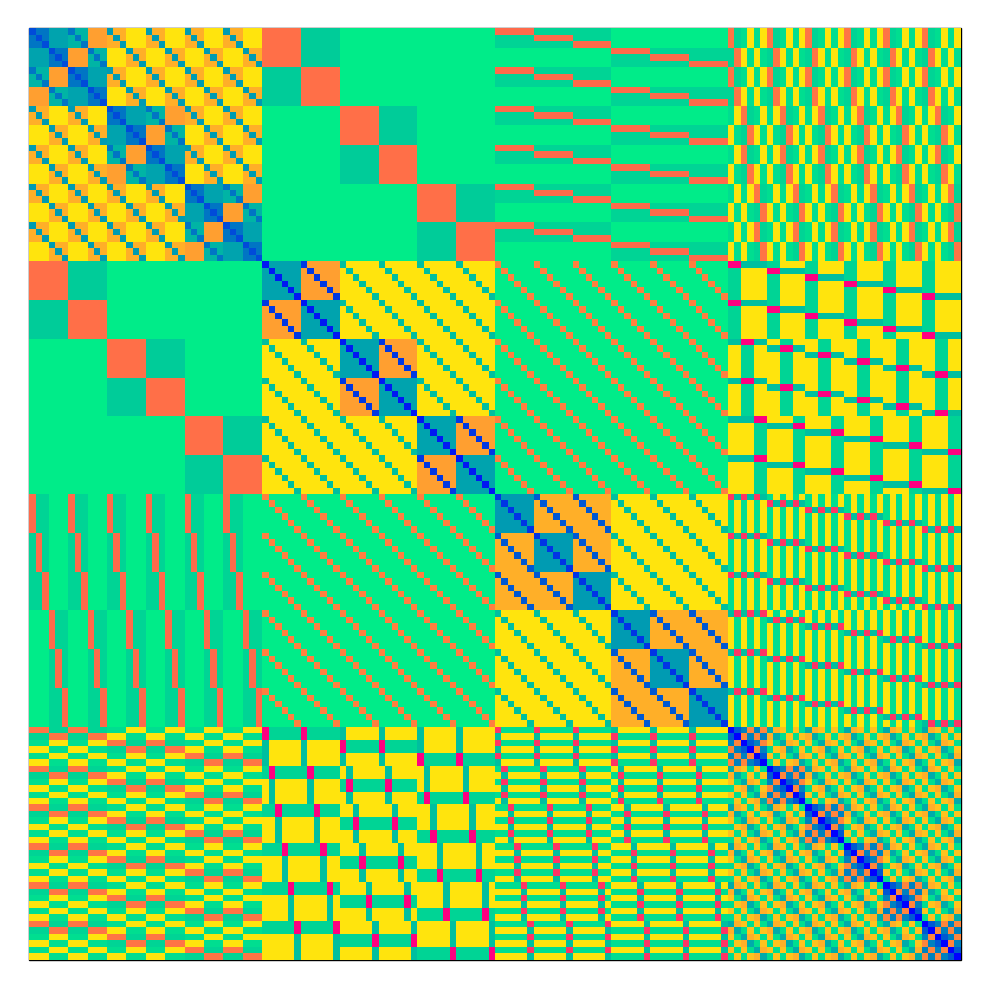}
~\includegraphics[scale=0.25]{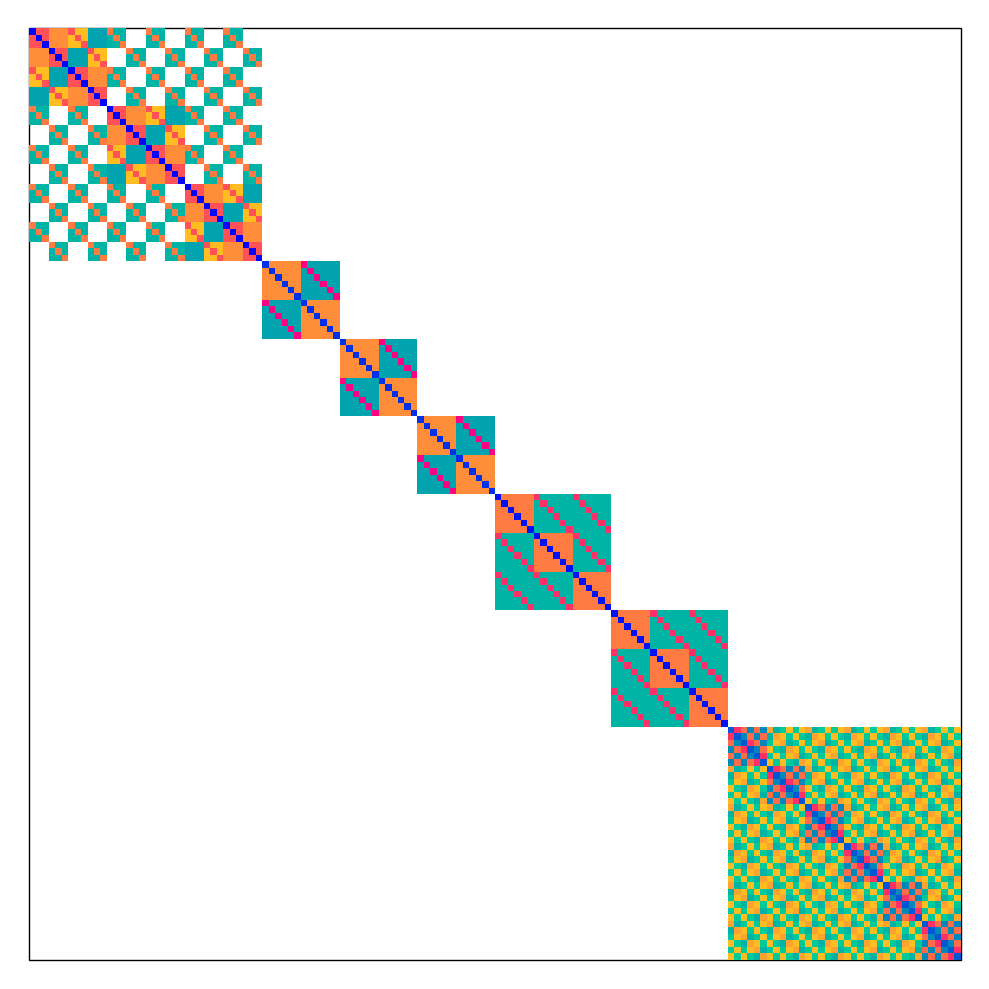}

\includegraphics[scale=0.25]{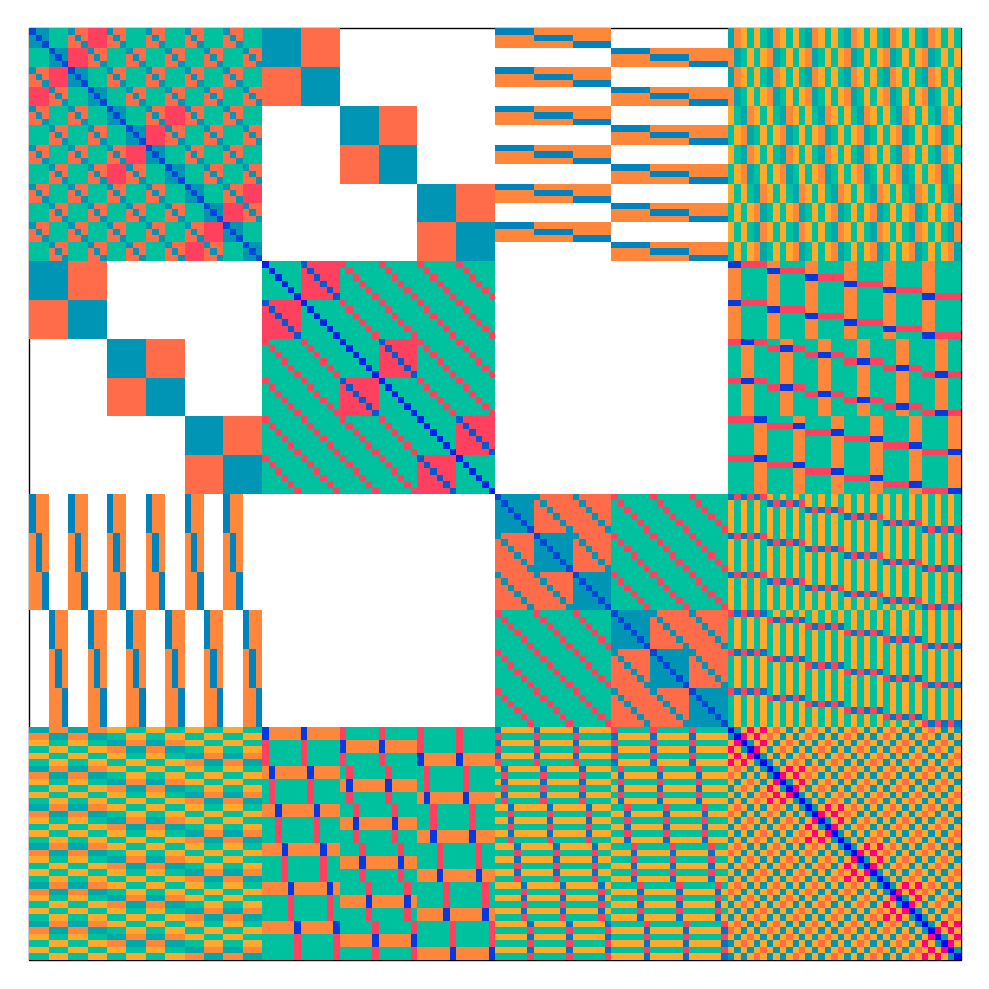}
~\includegraphics[scale=0.25]{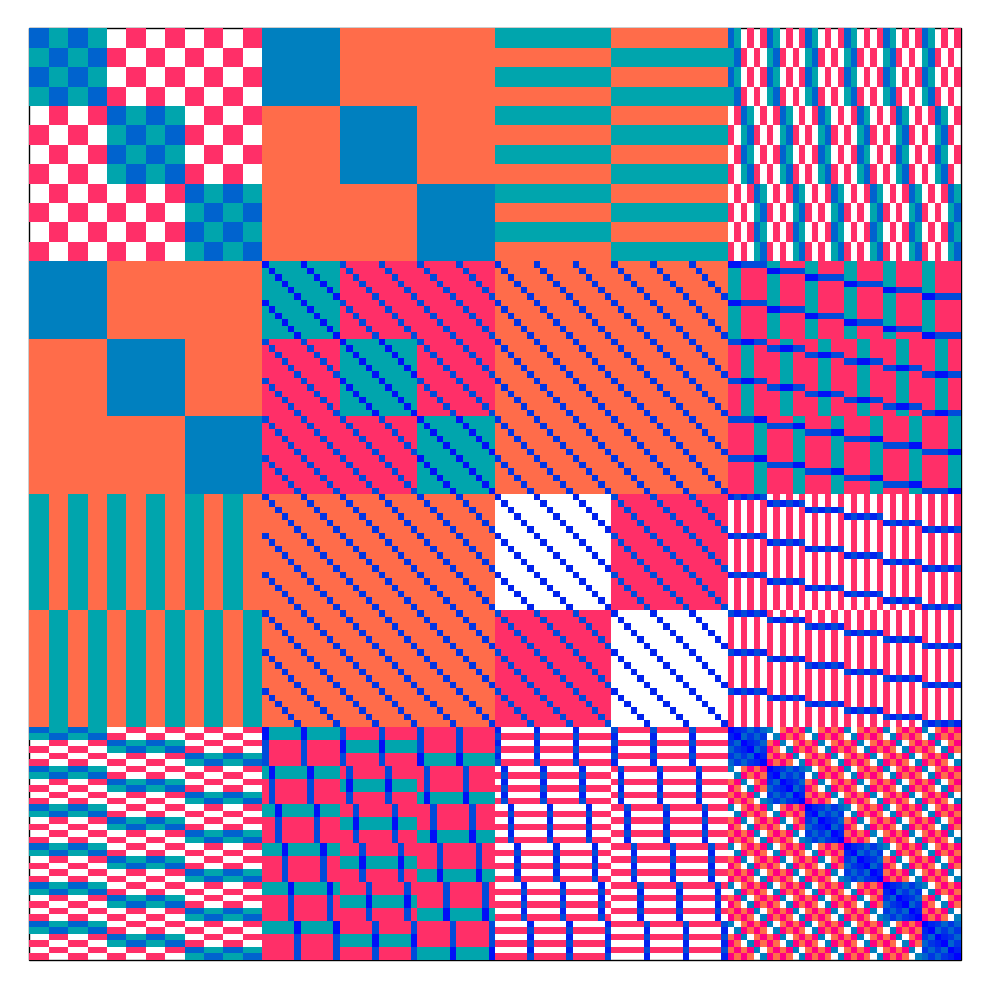}
\caption{Structure of entries of $E_0,E_1$ (top), $E_2,E_3$ (bottom) for $(h,w)=(2,3)$}
\label{fig:E0E1E2E3}
\end{center}
\end{figure}

%Let $E_0,\dots,E_4$ be the projection matrices in Proposition~\ref{prop:Mspec}.
The projectors can be computed as $E_i=V_i(V_i^\top V_i)^{-1}V_i^\top$, where $V_i$ is a matrix whose columns are a basis of eigenvectors for $\theta_i$.  As a special case, since $V_4$ is the all-ones vector, we have $E_4=\frac{1}{4n^2} J$.  The structure of entries for each of the other projectors is shown in Figure~\ref{fig:E0E1E2E3}.  Intensity of blue/red correspond respectively to extreme positive/negative entries, while shades of green/yellow correspond to positive/negative entries which are smaller in magnitude.

Knowing the eigenvalues and eigenspace projectors for $M$ can be used to compute a generalized inverse $M^+$ satisfying $MM^+M=M$. We explain this computation in the rest of this section.

The spectral decomposition of $M$ is given by $M=nE_1+2nE_2+3nE_3+4nE_4$. 
In what follows, $E_0$ will also be denoted $K$, since it projects onto the kernel of $M$.  Although $M$ itself is not invertible, if we take $\eta \neq 0$, say $\eta = n/x$, we can invert the additive shift $A=M+\eta K$ as
\begin{equation}
\label{eq:A-inv}
A^{-1} = \frac{1}{n} \left( x K +\sum_{j=1}^4 \frac{1}{j} E_j \right).
\end{equation}
This formula results from the $E_i$ being orthogonal idempotents with $E_0+E_1+\dots+E_4=I$.  Later on, to solve our linear system for fractional Sudoku, we make use of a generalized inverse $M^+$ of the form in \eqref{eq:A-inv}. It turns out that $x=3/2$, or $\eta=2n/3$, is a nice choice.  A discussion of this choice is given in the next subsection.

With some computer-assisted algebra, we found coefficients to express $A^{-1}$ in the basis $\{A_i:i=1,\dots,69\}$ for the adjacency algebra $\mathfrak{A}$. These are expressed in Table~\ref{tab:Ainv-coeffs}.  For convenience, we have cleared a denominator of $9n^3$ and then applied an additive shift of $5/16$.
Using our Sage \cite{Sage} worksheet at \url{https://github.com/pbd345/sudoku},
the interested reader can compute various
symbolic products in $\mathfrak{A}$, including a verification that Table~\ref{tab:Ainv-coeffs} does indeed give an inverse of $A$.
\begin{table}[htbp]
$$
\begin{array}{|l|r|l|r|}
\hline
\text{relations} & \text{coeffs} & \text{relations} & \text{coeffs} \\
\hline
1	&	 9n^2 + h + w	&	32	&	 9n^2 + n + h\\
2, 4, 5 	&	 h + w	&	34	&	 n + h\\
3, 6, 17, 19, 39, 43, 47, 51 	&	 w	&	38, 42 	&	 -9n/2 + h + w\\
7, 8, 33, 35, 40, 44, 55, 59 	&	 h	&	46, 50 	&	 -9nw/2 + n + w\\
10, 13 	&	 -9n/2 + w + 1	&	54, 58 	&	 -9nh/2 + n + h\\
11, 14 	&	 w + 1	&	62	&	 9n^2 + n + h + w - 1\\
12, 15, 24, 27, 29, 31 	&	1	&	63	&	 h + w - 1\\
16	&	 9n^2 + n + w	&	64	&	 n + w - 1\\
18	&	 n + w	&	65	&	 w - 1\\
20, 36, 48, 52, 56, 60 	&	 n	&	66	&	 n + h - 1\\
22, 25 	&	 -9n/2 + h + 1	&	67	&	 h - 1\\
23, 26 	&	 h + 1	&	68	&	 n - 1\\
28, 30 	&	 -7n/2 + 1	&	69	&	 -1\\
\hline
\end{array}$$
\caption{Coefficients of $9n^3A^{-1}+\frac{5}{16}J$}
\label{tab:Ainv-coeffs}
\end{table}

\subsection{Norm bounds}
\label{sec:inverse}

We work with the $\infty$-norm of vectors $\| \mathbf{x} \|_\infty = \max \{|x_1|,\dots,|x_n|\}$ and the induced norm on matrices
$$\| A \|_\infty = \max_i \sum_j |A_{ij}|.$$
It is straightforward to obtain a bound on the $\infty$-norm of \eqref{eq:A-inv} using the values in Tables~\ref{tab:degrees} and \ref{tab:Ainv-coeffs}.  The triangle inequality gives a crude bound of order $O(n^{-1})$, but we can get an exact value with the help of a computer. First, we store the coefficients of the projectors relative to our coherent configuration basis and make a note of their signs.  For each of the four sections corresponding to the edge types, we sum the absolute values of projector coefficients times the section row sums.  When we combine these as in 
\eqref{eq:A-inv}, the result is a list of three piecewise linear functions (one duplicate occurs for two sections), each multiplied by $n^{-1}$.  These functions are

\vspace{-11pt}
\begin{align*}
f_1(x)&=
3|\tfrac{x}{2} - \tfrac{3}{4}|+ 4|\tfrac{x}{6} - \tfrac{5}{36}|+ 2|\tfrac{x}{12} - \tfrac{7}{144}|+ 2|\tfrac{x}{12} - \tfrac{13}{144}|+ 2|\tfrac{x}{3} - \tfrac{11}{18}|+ 3|\tfrac{x}{2} - \tfrac{1}{4}|+ 1,\\
f_2(x)&=
2|\tfrac{x}{2} - \tfrac{3}{4}|+ 4|\tfrac{x}{6} - \tfrac{5}{36}|+ 2|\tfrac{x}{12} - \tfrac{7}{144}|+ 2|\tfrac{x}{12} - \tfrac{13}{144}|+ |\tfrac{x}{3} - \tfrac{11}{18}|+ |\tfrac{x}{3} - \tfrac{1}{9}|+ 2|\tfrac{x}{2} - \tfrac{1}{4}|+ 1,\\
f_3(x)&=
3|\tfrac{x}{2} - \tfrac{3}{4}|+ |\tfrac{x}{4} - \tfrac{25}{48}|+ 6|\tfrac{x}{6} - \tfrac{5}{36}|+ 3|\tfrac{x}{12} - \tfrac{13}{144}|+ 3|\tfrac{x}{3} - \tfrac{11}{18}|+ 3|\tfrac{x}{2} - \tfrac{1}{4}|+ 1.
\end{align*}
Graphs for the functions $f_i(x)$ are shown in Figure~\ref{fig:norm-pieces}.  It turns out that $\max\{f_i(x):i=1,2,3\}$ is minimized at $x=3/2$, yielding the dominant term $15/4n$, also an upper bound for all $h,w \ge 2$.
%, for $\|A^{-1}\|_\infty$ at $\eta=n/x=2n/3$.
The results of this computation are summarized in the lemma below.

\begin{lemma}
\label{lem:A-inv-bound}
Let $A=M+\tfrac{2n}{3}K$.  Then
$$\|A^{-1}\|_\infty =
\frac{15}{4n} -\frac{7(h+w)}{8n^2}-\frac{4}{9n^2}+\frac{31(h+w)-21}{72n^3}
< \frac{15}{4n}.$$
\end{lemma}

\begin{figure}
\begin{center}
\begin{tikzpicture}
\node at (0,0) {\includegraphics[scale=0.5]{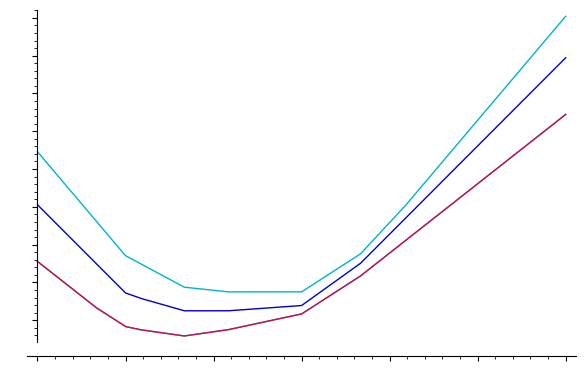}};
\node at (-1,-2.5) {$1$};
\node at (1.3,-2.5) {$2$};
\node at (3.6,-2.5) {$3$};
\node at (4.1,-2.2) {$x$};
\node at (3.9,2.35) {$f_3$};
\node at (3.9,1.75) {$f_2$};
\node at (3.9,1) {$f_1$};
\node at (-3.6,-1.64) {$3$};
\node at (-3.6,-1.15) {$4$};
\node at (-3.6,-0.66) {$5$};
\node at (-3.6,-0.17) {$6$};
\node at (-3.6,0.32) {$7$};
\node at (-3.6,0.81) {$8$};
\node at (-3.6,1.3) {$9$};
\node at (-3.65,1.79) {$10$};
\end{tikzpicture}
\caption{$A^{-1}$ norm plots $(f_1,f_2,f_3)$ as functions of $x:=n/\eta$}
\label{fig:norm-pieces}
\end{center}
\end{figure}

We also note the following bound on $K=E_0$.

\begin{lemma}
\label{lem:K-bound}
We have $\displaystyle \|K\|_\infty \le \frac{11}{2}-\frac{17(h+w)}{6n}+O(n^{-1})$.
\end{lemma}

\section{Perturbation}
\label{sec:perturb}

\subsection{Changes to $M$ resulting from pre-filled entries}
\label{sec:perturb1}

Let $S$ be a partial Sudoku of type $(h,w)$, where $hw=n$. Recall that $G_S$ is the graph obtained from $G_{hw}$ by deleting the edges of tiles corresponding to pre-filled entries in $S$.  Suppose $S$ has the $(1-\delta)n$ availability property.  That is, suppose every edge in $G_S$ is contained in at least $(1-\delta) n$ tiles in $G_S$.

Let $M=WW^\top$ and $M_S=W_SW_S^\top$, as introduced in Section~\ref{sec:setup}.  To set up our perturbation argument, we are interested in quantifying the change in $M$ resulting from pre-filling the entries of $S$.  It makes no sense to subtract $M_S$ from $M$ directly, since these matrices have different sizes.  However, we can use a convenient border.

Let $\widetilde{M}$ denote the $4n^2 \times 4n^2$ matrix, indexed by edges of $G$, whose entries are given by
\begin{equation}
\label{eq:M-wiggle}
\widetilde{M}(e,f) =
\begin{cases}
M_S(e,f) & \text{if } e,f \in E(G_S);\\
0 & \text{if } e\in E(G_S) \text{ and } f \not\in E(G_S);\\
M(e,f) & \text{if } e \not\in E(G_S).
\end{cases}$$
If we sort the rows and columns so that those indexed by $E(G_S)$ come first, then 
$$\widetilde{M} = \left[
\begin{array}{c|c}
M_S & O \\
\hline
\multicolumn{2}{c}{\text{as in }M}
\end{array} \right].
\end{equation}
Put $\Delta M = M-\widetilde{M}$.  We next estimate $\|\Delta M\|_\infty$ under our sparseness assumption.

For an edge $e \in E(G_S)$, let $U(e)$ denote the set of unavailable options
$$U(e) = \{ t \in T(G_{hw}) : e \in t \text{ and } f \in t \text{ for some } f \in E(G_{hw})\setminus E(G_S)\}.$$
Put $u(e)=|U(e)|$ and $\mathbf{u} = (u(e):e \in E(G_S))$.  In more detail, if $e$ is an edge of type
row-column, say $e=\{r_i,c_j\}$, then $U(e)$ keeps track of those symbols $k$ which are not able to be placed in 
cell $(i,j)$ because $k$ already appears in row $i$ or column $j$ or box $\mathrm{box}(i,j)$.  If $e$ is an edge of type row-symbol, say $e=\{r_i,s_k\}$, then $U(e)$ keeps track of those columns $j$ which are unavailable for symbol $k$ in row $i$, either because cell $(i,j)$ was pre-filled or $k$ appears somewhere else in column $j$ or box $\mathrm{box}(i,j)$. Note that several columns might be eliminated as options if $k$ appears in a box intersecting row $i$. Edges of type column-symbol behave in an analogous way.  Finally, if $e$ is an edge of type box-symbol, say $e=\{b_\ell,s_k\}$, then $U(e)$ keeps track of those cells $(i,j)$ in box $\ell$ for which $k$ is not allowed, either because $(i,j)$ was already filled in $S$, or because $k$ already appears in the row or column bundle for box $\ell$.

\begin{lemma}
\label{lem:DeltaM-u}
We have $\mathbf{0} \le \Delta M \mathds{1} \le 4\mathbf{u}$ entrywise.  In particular, $\| \Delta M \|_\infty \le 4\|\mathbf{u}\|_\infty$.
\end{lemma}

\begin{proof}
Entry $e$ of $(\Delta M) \mathds{1}$ equals $\sum_f \Delta M(e,f)$.  The summand is the number of unavailable tiles $t$ with $e \in t$ and $f \in t$.  Since each copy $t$ contains four edges, this count is at most $4 u(e)$.
\end{proof}

Now, if $S$ has the $(1-\delta)$-availability property, then by definition we have 
$\|\mathbf{u}\|_\infty \le \delta n$.  And recall that if $S$ has the $\epsilon$-dense property, then it has the $(1-3\epsilon)$-availability property, as explained at the end of Section~\ref{sec:decomp}.  These, together with Lemma~\ref{lem:DeltaM-u} immediately give bounds on $\Delta M$.

\begin{lemma}
\label{lem:DeltaM-sparse}
With $\Delta M$ constructed from $S$ as above, we have
\vspace{-11pt}
\begin{enumerate}
\item
$\| \Delta M \|_\infty \le 4 \delta n$ if $S$ has the $(1-\delta)$-availability property; and
\item
$\| \Delta M \|_\infty \le 12 \epsilon n$ if $S$ is $\epsilon$-dense;
\end{enumerate}
\end{lemma}

\subsection{A guarantee on nonnegative solutions}
\label{sec:perturb2}

The following can be distilled from \cite[Section 3]{BD}.

\begin{lemma}
\label{lem:perturb}
Let $A$ be an $N \times N$ invertible matrix over the reals.  Suppose $A-\Delta A$ is a perturbation.   Then
\vspace{-10pt}
\begin{enumerate}
\item
$A-\Delta A$ is invertible provided $\|A^{-1} \Delta A \|_\infty <1$; and
\item
the solution $\vx$ to $(A-\Delta A)\vx = A \mathds{1}$ is entrywise nonnegative provided $\|A^{-1} \Delta A\|_\infty \le \frac{1}{2}$.
\end{enumerate}
\end{lemma}

Lemma~\ref{lem:perturb} can be proved using the expansion $(A-\Delta A)^{-1}=\sum_{k=0}^\infty (A^{-1} \Delta A)^k A^{-1}$.  More details on matrix norms and the convergence of this series can be found in Horn and Johnson's book \cite{MatrixAnalysis}.

We would like to apply Lemma~\ref{lem:perturb} to the perturbation $M-\Delta M$, but we must take care to handle the nontrivial kernel.  Of various possible approaches, one convenient thing to do is to place those columns of $K=E_0$ corresponding to non-edges of $G_S$ in the perturbation.
In more detail, let $A=M+\eta K$ and observe that $A \mathds{1}=4n \mathds{1}$.  That is, for this choice of $A$, the right side of the system in Lemma~\ref{lem:perturb} is just a scalar multiple of the all-ones vector. Define $\Delta A = \Delta M + \eta K'$, where
$$K’(e,f)=\begin{cases} 0 & \text{if } f \in E(G_S);\\
K(e,f) & \text{otherwise}.
\end{cases}$$

Note that
\begin{equation}
\label{eq:ainv-kprime}
A^{-1} (\eta K’) = \left(\frac{1}{\eta} K + \sum_{j=1}^4 \frac{1}{jn} E_j \right) (\eta K’) =  K’, 
\end{equation}
since the columns of $K'$ are orthogonal to each of the other eigenspaces.

\begin{lemma}
\label{lem:Kprime-sparse}
Suppose $S$ is $\epsilon$-dense. Then, for large $h$ and $w$, $\| K’ \|_\infty \le (\epsilon+o(1)) \|K\|_\infty$.
\end{lemma}

\begin{proof}
Write $K=\sum_{i=1}^m c_i A_i$.  Fix $e \in E(G_{hw})$.
Then we have
$$\sum_{f \in E(G_{hw})} |K(e,f)| = 
\sum_i |c_i| d_i(e),$$
where $d_i(e)$ is the number of edges $f$ with $(e,f) \in R_i$.  Recall that $d_i(e)$ is zero unless $e$ is of an edge type corresponding to the first coordinate of $R_i$, and we may assume a canonical choice $r_1c_1$, $r_1s_1$, $c_1s_1$, or $s_1 b_1$ for $e$.  

%We applied our computational methods in $\mathfrak{A}$ as before to find the leading terms in
%the product $c_i d_i(e)$ for each relation index $i$.  The results are shown in Table~\ref{tab:K-est}, divided into columns %according to the edge type of $e$. 

Let $\overline{G_S}$ denote the complement of $G_S$ in 
$G_{hw}$.  Then we have
\begin{equation}
\label{eq:Kprime-estimate}
\sum_{f \in E(G_{hw})} |K'(e,f)| = 
\sum_{f \in E(\overline{G_S})} |K(e,f)| = 
\sum_{i=1}^m |c_i| d'_i(e),
\end{equation}
where $d'_i(e)$ is the number of edges $f \in E(\overline{G_S})$ with $(e,f) \in R_i$.  With the exception of $i \in I :=\{1,2,4,16,32,62\}$, each relation $R_i$ has an associated 
feature which, owing to our $\epsilon$-density assumption, limits the number of missing edges $f$ in $G_S$ with $(e,f) \in R_i$.
These features are indicated in Table~\ref{tab:K-est}, along with bounds on leading terms of $|c_i|d_i'(e)$.  A legend and upper bound on corresponding $d'_i(e)$ are given in Table~\ref{tab:sparse-legend}.  Terms with $i \in I$ are of lower order.  Otherwise, when we compute the sum \eqref{eq:Kprime-estimate}, we obtain the same leading terms as in the computation of $\|K\|_\infty$, each times $\epsilon$.  The edge type with largest total coefficient of $\epsilon$ is the box-symbol type, or column 4 in Table~\ref{tab:K-est}.  This results in
$$\|K'\|_\infty \le \left(\frac{11}{2}+O(h^{-1}+w^{-1}) \right) \epsilon + \frac{h+w}{2n} + O(h^{-2}+w^{-2}+n^{-1}).\hfill \qedhere$$
\end{proof}

\begin{table}[htbp]
\small
\begin{tabular}{|rrr|rrr|rrr|rrr|}
\hline
& leading  & sparse	& & leading  & sparse	& & leading  & sparse	& & leading  & sparse \\
$i$	& term & feature	& $i$	& term & feature	&$i$	& term & feature	&$i$	& term & feature\\
\hline
1	&	$ 3/2n $	&	-	&	13	&	$ \epsilon/2 $	&	r	&	25	&	$ \epsilon/2 $	&	c	&	42	&	
$\epsilon/2 $	&	b	\\
2	&	$ 1/h $	&	-	&	14	&	$ \epsilon/6 $	&	rb	&	26	&	$ \epsilon/6 $	&	cb	&	43	&	
$ \epsilon/6 $	&	rb	\\
3	&	$ \epsilon/2 $	&	r	&	15	&	$ \epsilon/12 $	&	all	&	27	&	$ \epsilon/12 $	&	all	&	44	&	$ \epsilon/6 $	&	cb	\\
4	&	$ 1/w $	&	-	&	16	&	$ 1/2h $	&	-	&	30	&	$ \epsilon/3 $	&	s	&	45	&	$ \epsilon/12 $	&	all	\\
5	&	$ \epsilon/2 $	&	b	&	17	&	$ \epsilon/2 $	&	r	&	31	&	$ \epsilon/12 $	&	all	&	50	&	$\epsilon/2 $	&	srb	\\
6	&	$ \epsilon/3 $	&	rb	&	18	&	$ \epsilon/2 $	&	srb	&	32	&	$ 1/2w $	&	-	&	51	&	$ \epsilon/6 $	&	rb	\\
7	&	$ \epsilon/2 $	&	c	&	19	&	$ \epsilon/3 $	&	rb	&	33	&	$ \epsilon/2 $	&	c	&	52	&	$ \epsilon/6 $	&	s	\\
8	&	$ \epsilon/3 $	&	cb	&	20	&	$ \epsilon/6 $	&	s	&	34	&	$ \epsilon/2 $	&	scb	&	53	&	$ \epsilon/12 $	&	all	\\
9	&	$ \epsilon/12 $	&	all	&	21	&	$ \epsilon/12 $	&	all	&	35	&	$ \epsilon/3 $	&	cb	&	58	&	$\epsilon/2 $	&	scb	\\
10	&	$ \epsilon/2 $	&	r	&	28	&	$ \epsilon/3 $	&	s	&	36	&	$ \epsilon/6 $	&	s	&	59	&	$ \epsilon/6 $	&	cb	\\
11	&	$ \epsilon/6 $	&	rb	&	29	&	$ \epsilon/12 $	&	all	&	37	&	$ \epsilon/12 $	&	all	&	60	&	$ \epsilon/6 $	&	s	\\
12	&	$ \epsilon/12 $	&	all	&	46	&	$ \epsilon/2 $	&	srb	&	54	&	$ \epsilon/2 $	&	scb	&	61	&	$ \epsilon/12 $	&	all	\\
22	&	$ \epsilon/2 $	&	c	&	47	&	$ \epsilon/6 $	&	rb	&	55	&	$ \epsilon/6 $	&	cb	&	62	&	$ (h+w)/2n $	&	-	\\
23	&	$ \epsilon/6 $	&	cb	&	48	&	$ \epsilon/6 $	&	s	&	56	&	$ \epsilon/6 $	&	s	&	63	&	$ \epsilon/2 $	&	b	\\
24	&	$ \epsilon/12 $	&	all	&	49	&	$ \epsilon/12 $	&	all	&	57	&	$ \epsilon/12 $	&	all	&	64	&	$ \epsilon/2 $	&	srb	\\
38	&	$ \epsilon/2 $	&	b	&		&		&		&		&		&		&	65	&	$ \epsilon/3 $	&	rb	\\
39	&	$ \epsilon/6 $	&	rb	&		&		&		&		&		&		&	66	&	$ \epsilon/2 $	&	scb	\\
40	&	$ \epsilon/6 $	&	cb	&		&		&		&		&		&		&	67	&	$ \epsilon/3 $	&	cb	\\
41	&	$ \epsilon/12 $	&	all	&		&		&		&		&		&		&	68	&	$ \epsilon/3 $	&	s	\\
	&		&		&		&		&		&		&		&		&	69	&	$ \epsilon/4 $	&	all	\\
\hline
\end{tabular}
\normalsize
\caption{Terms contributing to $\|K\|_\infty$}
\label{tab:K-est}
\end{table}

\begin{table}[htbp]
\begin{tabular}{|rlr|rlr|}
\hline
 & sparse feature &  bound &  & sparse feature &  bound\\
\hline
r & cells filled in a row & $\epsilon n$ & b & cells filled in a box & $\epsilon n$\\
c & cells filled in a column & $\epsilon n$ & s & occurrences of a symbol & $\epsilon n$\\
rb & cells filled in a row bundle & $\epsilon nh$ & srb & times a symbol is in a row bundle & $\epsilon h$\\
cb & cells filled in a column bundle & $\epsilon nw$ & scb & times a symbol is in a column bundle & $\epsilon w$\\
all & cells filled in all of $S$ & $\epsilon n^2$ & & &\\
\hline
\end{tabular}
\normalsize
\caption{Legend for Table~\ref{tab:K-est} and $\epsilon$-density bounds}
\label{tab:sparse-legend}
\end{table}

Putting together Lemmas~\ref{lem:A-inv-bound}, \ref{lem:K-bound}, \ref{lem:DeltaM-sparse} and \ref{lem:Kprime-sparse}, we obtain a bound on $A^{-1} \Delta A$.

\begin{prop}
\label{prop:Ainv-DeltaA}
Suppose $S$ is an $\epsilon$-dense partial Sudoku of type $(h,w)$ where $h,w$ are large.  Then
$\| A^{-1} \Delta A \|_\infty  < 101\epsilon/2+o(1).$
\end{prop}

\begin{proof}
From \eqref{eq:ainv-kprime}, submultiplicativity and the triangle inequality, 
$$\| A^{-1} \Delta A \|_\infty \le \| A^{-1}\|_\infty \|\Delta M \|_\infty + \|K' \|_\infty < \tfrac{15}{4n} \times 12\epsilon n+ \tfrac{11}{2} \epsilon + o(1) = \tfrac{101}{2} \epsilon + o(1).\hfill \qedhere$$
\end{proof}

\subsection{Proof of the main result}
\label{sec:proof}

We are now ready to prove our result on partial Sudoku completion under the $\epsilon$-dense assumption.

\begin{proof}[Proof of Theorem~\ref{thm:main}]
Apply  Lemma~\ref{lem:perturb} to $A$ and $\Delta A$ constructed as above.  
Under the assumption $\epsilon<1/101$, Proposition~\ref{prop:Ainv-DeltaA} gives 
$\|A^{-1} \Delta A \|_\infty  < 1/2$ for sufficiently large $h,w$. This implies an entrywise nonnegative solution to $(A-\Delta A) \mathbf{x} = \mathds{1}$. Let $\mathbf{x}'$ denote the restriction of $\mathbf{x}$ to $E(G_S)$.
Since $A-\Delta A$ is block lower-triangular with respect to the partition into edges and non-edges of $G_S$, it follows that $(M_S+ \eta K[S]) \mathbf{x}' = \mathds{1}$.  We note that $M_S$ and $K[S]$ are symmetric and satisfy the conditions in Proposition~\ref{prop:Kres-MG}.  Therefore, Lemma~\ref{lem:shift} implies $M_S \mathbf{x}' = \mathds{1}$. This, in turn, implies a nonnegative solution to the linear system for completing $S$ via the coefficient matrix $W_S$.
\end{proof}

It is worth a remark that the lower order terms in Lemmas~\ref{lem:A-inv-bound} and \ref{lem:K-bound} are actually negative.  This means our hypothesis of large $h$ and $w$ is only really used to control the mild lower-order terms in $K'$.  In general, our method is robust for small partial Sudoku, often succeeding in practice with densities much larger than $1/101$.  For instance, the completion shown in Figure~\ref{fig:frac-sud-example} came from applying the above proof method.

\section{Thin boxes}
\label{sec:thin}

In this section, we investigate in more detail the case of Sudoku of type $(h,w)$ with fixed width $w$ and height $h=n/w$ where $n$ is a large multiple of $w$.  In Section~\ref{sec:intro}, we observed that there exist non-completable partial $n \times n$ Sudoku latin squares with no row, column, symbol, or box used very often; indeed, this motivated our row/column bundle condition for $\epsilon$-density.  The idea is that symbol $k$ can be forced in an entry $(i,j)$ using only $O(h+w)$ pre-filled occurrences of $k$ in the row and column bundle containing $(i,j)$.  Now, when $w$ is fixed and $h=n/w$, this construction requires a symbol be used $O(n)$ times.  One could hope that all partial Sudoku of type $(n/w,w)$ which are $\epsilon$-dense in the sense of latin squares could still admit a completion, even without the column bundle condition.

Notice that this would not follow directly from our earlier work; for instance if $w=2$, any cell pre-filled with symbol $k$ in the the top-left box prohibits the placement of $k$ in the same column of the bottom-left box.  This eliminates $n/2$ options for placing $k$ in that box, resulting in a perturbation which is far too severe for our methods of Section~\ref{sec:perturb}.

The next result is a general construction giving a barrier to (fractional) completion in the context of fixed $w$.  It is similar to the second example in Figure~\ref{fig:no-completion}.
%and is strictly stronger than the Daykin-H\"aggkvist conjectured threshold of $1/4$ for latin squares.

\begin{prop}
\label{prop:dh-thin}
Let $w$ be an integer, $w \ge 2$.  For any $\epsilon>1/3w$, there exists a partial Sudoku of type $(h,w)$, with no completion and such that every row, column, symbol, and box is used at most $\epsilon hw$ times.
\end{prop}

\begin{proof}
Suppose $h \ge w$.  Put $a=\lceil (h+w)/3 \rceil$ and $n=hw$.  Take $A,A'$ as disjoint sets of $a$ symbols, which is possible since $n=hw \ge 2a$ holds under the assumption $h \ge w \ge 2$.  Let $\mathcal{L}$ be the leftmost box bundle; that is, $\mathcal{L}$ consists of the first $w$ columns and box numbers $hj+1$ for $j=1,\dots,w$.
Define a partial Sudoku $S$ of type $(h,w)$ as follows: (a)
put the elements of $A'$ in entries $(i,1)$, $i=1,\dots,a$;  
(b)
put the elements of $A$ in column $j$ of the $j$th box of $\mathcal{L}$ for each $j=2,\dots,w$;
(c) 
put the elements of $A$ strictly to the right of $\mathcal{L}$ in row $i$ for $i=a+1,\dots,2a-w+1$.   Several remarks are needed.  First, each of the placements of (a) and (b) fit within the respective boxes because $a \le h$ holds under the assumption $h \ge w \ge 2$.
Next, we claim that (c) can be done in such a way that no $h \times w$ box has repeated elements.  To achieve this, we can order the elements of $A$ somehow in first of these rows, starting at column $w+1$, and then shift this same ordering to the left by multiples of $w$ in each successive row.  To check that there is room for this, we note that $(a-w)w+a \le n-w$.  This holds with equality for $h=w=a=2$ and otherwise can be easily verified with the estimate $a \le (h+w+2)/3$.

Now, by construction, every row, column and box contains at most $a$ elements.  Each symbol of $A$ occurs $(w-1)+(a-w+1)=a$ times and each symbol of $A'$ occurs exactly once.
Given a hypothetical fractional completion of $S$, consider where the elements of $A$ would occur in the first box.  Columns $2,\dots,w$ are not available by (b), and rows $1,\dots,2a-w+1$ of the first column are blocked by (a) and (c).  So, the elements of $A$ must fit into the entries $(i,1)$ for $2a-w+1 < i \le h$.  But there are only $h+w-2a-1<a$ such entries.  This is a contradiction, and hence $S$ has no fractional completion.
 
Finally, note that by taking $h$ sufficiently large, we can ensure $a/n<\epsilon$.  This completes the proof.
\end{proof}

Proposition~\ref{prop:dh-thin} shows that, in the case of thin boxes with fixed $w$, any result guaranteeing fractional completion with up to $\epsilon n$ occurrences of a row, column, symbol, or box (ignoring row bundle density) must have $\epsilon$ being a function of $w$.  With some minor adaptation, our methods can produce a result of this form.
First, though, it is helpful to have a technical lemma on adding entries to latin rectangles.  Let us say that an $m \times n$ partial latin rectangle is $\epsilon$-\emph{dense} if every row has at most $\epsilon n$ filled cells, every column has at most $\epsilon m$ filled cells, and every symbol is used at most $\epsilon m$ times.

\begin{lemma}
\label{lem:extension}
Suppose $0<\epsilon,\delta < \frac{1}{6}$.
Let $n \ge m$ and suppose we are given an $\epsilon$-dense $m \times n$ partial latin rectangle $P$. Let $A_1,\dots,A_n \subset [n]$ with $|A_j| < \delta m$ for each $j$ and $|\{j:k \in A_j\}| < \delta m$ for each $k$.  Then $P$ is contained in a 
$3(\delta+\epsilon)$-dense $m \times n$ partial latin rectangle $P'$ such that, for each $j=1,\dots,n$, column $j$ of $P'$ contains the elements of $A_j$.
\end{lemma}

\begin{proof}
We construct $P'$ from $P$ by adding symbols from $A_j$ one column at a time.  It can be assumed that $A_j$ is disjoint from the set of symbols already in column $j$ of $P$ by replacing $A_j$ with a smaller set. Suppose we have extended $P$ by $j-1$ columns for $1 \le j \le n$, and let $P_j$ be the resulting array.  We may sort the rows of $P_j$ in weakly increasing order of how many symbols are in each.  Let $B_j$ consist of the first 
$\lceil (\epsilon+2\delta)m \rceil$ row indices; those correspond to rows with the fewest number of symbols.  Let $G_j$ denote the bipartite graph with vertex partition $A_j,B_j$ and an edge drawn between $k \in A_j$ and $r \in B_j$ if and only if symbol $k$ is available to be placed in row $r$ of column $j$.  We claim that $G_j$ has a matching that uses every element of $A_j$.  Consider symbol $k$ in $A_j$.  It can be placed in any row that doesn't already have symbol $k$.  But $k$ appears in at most $\epsilon m$ rows of $P$ and was previously added at most $\delta m$ times in forming $P_j$.  Therefore,
$$\deg_{G_j}(k) \ge (\epsilon+2\delta)m-\epsilon m - \delta m = \delta m > |A_j|.$$
So our claim of the existence of the matching follows by Hall's Theorem.  Let $P'$ be the resulting array after extending $P_n$.  In checking that $P'$ is $3(\epsilon+\delta)$-dense, the column condition and symbol condition follow immediately from $P$ being $\epsilon$-dense and the hypotheses on the sets $A_j$.  Suppose row $r$ contains more than $3(\epsilon+\delta)n$ symbols.  By choice of $B_j$, when the last symbol was added to row $r$, there were at least $\lfloor(1-\epsilon-2\delta)m \rfloor$ other rows with at least $3(\delta+\epsilon)n-1 \ge 2(\delta+\epsilon)n$ symbols. (Here, we used $\epsilon n \ge 1$ since otherwise $P$ is empty.)  These rows, along with row $r$, would account for at least $\frac{m}{2} \times 2n(\delta+\epsilon) = mn(\delta+\epsilon)$ filled entries.  But this exceeds the total number of filled entries in $P'$, leading to a contradiction.
\end{proof}

We are now ready for a result on fractional completion for the case of thin boxes without a density assumptions on bundles.  The idea is to use Lemma~\ref{lem:extension} to first add symbols with low availability to boxes in the same column bundle, producing a partial Sudoku with the $(1-\delta)$-availability property.  This can then be completed as before using the estimates in  Lemmas~\ref{lem:DeltaM-sparse} and \ref{lem:Kprime-sparse}.  

\begin{prop}
For each $w \ge 2$, there exists a constant $\epsilon=\epsilon(w)$ such that every sufficiently large partial Sudoku of type $(n/w,w)$ in which every row, column, symbol and box is used at most $\epsilon n$ times admits a fractional completion.
\end{prop}

\begin{proof}
Let $S$ be an $\epsilon$-dense partial Sudoku of type $(n/w,w)$ and let $P$ be the restriction of $S$ to the top row bundle.  Then $P$ is an $m \times n$ partial latin rectangle, where $m=n/w$, and $P$ is $w \epsilon$-dense.  Consider the column bundles of $S$.  For the $i$th column bundle, let $Z_i$ be the set of symbols which occur in this bundle outside of $P$.  Our strategy is to add the symbols of $Z_i$ to the $i$th column bundle of $P$.  Let $A_j$, $j=1,\dots,n$, be any sets of symbols whose disjoint union over the $i$th bundle equals $Z_i$ for each $i$.  Note that $|A_j| \le |Z_i| \le \epsilon nw =\epsilon mw^2$ and every symbol appears in at most $\epsilon n = \epsilon mw$ of the $A_j$.  By Lemma~\ref{lem:extension}, $P$ can have the entries of $A_j$ added to column $j$ in such a way that the resulting latin rectangle $P'$ is $O(w^2) \epsilon$-dense.  Let us carry out the same procedure for each of the $w$ row bundles of $S$.  This produces a partial Sudoku $S'$ with the following properties:
(1) every row, column, symbol and box is used at most $O(w^3) \epsilon n$ times in $S'$; (2) for every box $b$ of $S'$, the symbols that occur in the column bundle containing $b$ also occur within $b$ itself.  In particular, (2) says that a column-symbol edge $\{c_j,s_k\} \in  E(G_{S'})$ loses no options for tiles because of symbol $k$ occurring in a different box of the column bundle containing $c_j$.  Similarly, a box-symbol edge $\{b_\ell,s_k\} \in  E(G_{S'})$ loses no options for tiles.  This means $S'$ has the $(1-\delta)$-availability property, where $\delta =  C_1\epsilon$ for some constant $C_1$ depending only on $w$. 
Let us construct the matrix $\Delta M$ based on $S'$ as in Section~\ref{sec:perturb1}. By Lemma~\ref{lem:DeltaM-sparse}, we have
$\|\Delta M \|_\infty < 4 C_1 \epsilon n$.   We make one minor change to the perturbation set-up of Section~\ref{sec:perturb2}.  Define $\Delta A = \Delta M + \eta K''$, where
$$K’'(e,f)=\begin{cases} 0 & \text{if } e \not\in E(G_{S'}) \text{ or } f \in E(G_{S'});\\
K(e,f) & \text{otherwise}.
\end{cases}$$
For $\|K''\|_\infty$, it is straightforward to adapt Lemma~\ref{lem:Kprime-sparse} to get a weaker upper bound of the form $\|K''\|_\infty < C_2 \epsilon$, where $C_2$ depends on $w$.  For this, a few additional remarks are useful.  First, $K''$ vanishes on the diagonal, so receives no contribution from relations $R_i$, $i \in \{1,16,32,62\}$.  Next, $K''$ also vanishes on entries supported by those relations $R_i$, $i \in \{34,54,58,66\}$ having `scb' constraints.  This is because $(e,f) \in R_i$ with $e \in E(G_{S'})$ implies $f \in E(G_{S'})$ by our construction of $S'$.  Finally, since $w$ is fixed and $h$ is large, the number of nonzero entries of $K''$ supported by relation $R_4$ is bounded by $\epsilon h$ times a function of $w$, from sparsity of each column of $S'$.  Bounds on contributions from the other relations are as in Table~\ref{tab:K-est}.

Letting $A=M+\eta K$, for (say) $\eta=n$, the bound in Lemma~\ref{lem:A-inv-bound} gives $\|A^{-1}\|_\infty < C_3 n^{-1}$.  Choose $\epsilon(w)=1/2C_3(4C_1+C_2)$.  Then with this choice of $\epsilon$, we have
$$\| A^{-1} \Delta A \|_\infty \le \| A^{-1}\|_\infty (\|\Delta M \|_\infty + \|n K''\|_\infty) <  C_3(4C_1+C_2)\epsilon = \tfrac{1}{2}.$$  
It follows from Lemma~\ref{lem:perturb} that
the system $(A-\Delta A) \mathbf{x} = \mathds{1}$ has a nonnegative solution $\mathbf{x}$, and we finish the proof as before by restricting $\mathbf{x}$ to the edges of $G_S$.
\end{proof}

For simplicity, we have omitted explicit estimates on the constants $C_i$.  Probably the constant $C_1$ related to our construction of $S'$ is the main place where an improvement could be made.

\section{Variations and concluding remarks}
\label{sec:vars}

Suppose we generalize our setting so that each Sudoku box/cage is an arbitrary polyomino of $n$ cells.  Most of our set-up stays the same, except that the $n$-to-$1$ function $\mathrm{box}(i,j)$ mapping cells to boxes changes, say to $\mathrm{box}'(i,j)$.  If this change is sufficiently small, we can reasonably expect the same perturbation methods to give a fractional completion guarantee for sparse partial Sudoku of this generalized type.

To add some precision, let us define a polyomino Sudoku as above to have $\alpha$-\emph{approximate} type $(h,w)$ if, for each box $\ell$, the symmetric difference between $\mathrm{box}^{-1}(\ell)$ and $(\mathrm{box'})^{-1}(\ell)$ can be covered by $\alpha h$ rows and $\alpha w$ columns. The $0$-approximate case coincides with our standard setting of rectangular boxes.
Let $S$ be an $\alpha$-approximate partial Sudoku and define the matrix $M_S$ as before; that is, for unused edges $e$ and $f$, $M_S(e,f)$ equals the number of available tiles containing both $e$ and $f$.
Let $M'$ denote the $4n^2 \times 4n^2$ matrix for the empty Sudoku with polyomino boxes defined by $\mathrm{box}'$, and let $M$ be our usual matrix for the case of $h \times w$ boxes.  We note a few observations on $M'$:

\vspace{-11pt}
\begin{itemize}
\item
$M'$ agrees with $M$ on 
\begin{itemize}
\item
diagonal entries (each equals $n$);
\item
all entries indexed by edges of type row-column, row-symbol, or column-symbol (since boxes are not involved);
\end{itemize}
\item
if $\{e,f\}=\{r_is_k,b_\ell s_k\}$, then $M'(e,f)$ counts the cells shared between row $i$ and box $\ell$;
\item
if $\{e,f\}=\{c_js_k,b_\ell s_k\}$, then $M'(e,f)$ counts the cells shared between column $j$ and box $\ell$;
\item
if $\{e,f\}=\{r_ic_j,b_\ell s_k\}$, then $M'(e,f)=1$ if $\mathrm{box}(i,j)=\ell$ and $M'(e,f)=0$ otherwise.
\end{itemize}
Under the $\alpha$-approximate assumption, the above gives
\begin{equation}
\label{eq:alpha}
\|M-M'\|_\infty \le 2(\alpha h)w+2(\alpha w)h = 4\alpha n.
\end{equation}
Assume $S$ has the $(1-\delta)$-availability property.  Construct $\widetilde{M}$ using $M_S$ and $M$ as in \eqref{eq:M-wiggle}.  From \eqref{eq:alpha}, Lemma~\ref{lem:DeltaM-sparse}, and the triangle inequality,
\begin{equation}
\label{eq:triangle-ineq}
\|M-\widetilde{M}\|_\infty \le \|M-M'\|_\infty + \|M'-\widetilde{M}\|_\infty < 4(\alpha+\delta)n.
\end{equation}
Plugging this into the perturbation methods used earlier, we are able to get a variant on our main result for the approximate rectangular setting.  For brevity, we give a conservative statement without explicit constants.  Here, the $\epsilon$-density condition on row/column bundles should be taken as an adaptation of that for rectangular boxes.
%The kernel of $M'$, while very similar to that of $M$, may introduce additional dependence on $\alpha$ beyond that in 
%\eqref{eq:triangle-ineq}.

\begin{prop}
There exist positive constants $\alpha$ and $\epsilon$ such that every $\epsilon$-dense $\alpha$-approximate partial Sudoku with large $h$ and $w$ has a fractional completion. 
\end{prop}

It is possibly of interest to consider properties of the matrix $M'$ for specific box arrangements.  A `Pentadoku' is a $5 \times 5$ Sudoku-like puzzle whose cages are pentomino shapes.  Each cage (in addition to each line) must contain the numbers from $1$ to $5$ exactly once. Figure~\ref{fig:pentadoku} shows an example of a completed Pentadoku puzzle.

\begin{figure}[htbp]
\begin{minipage}[c]{0.45\linewidth}
\begin{center}
\begin{tikzpicture}[scale=0.8]
\foreach \a in {0,1,...,5}
                \draw (\a,0)--(\a,5);
\foreach \a in {0,1,...,5}
                \draw (0,\a)--(5,\a);
\foreach \a in {0,5}
                \draw[line width=2pt] (\a,0)--(\a,5);
\foreach \a in {0,5}
                \draw[line width=2pt] (0,\a)--(5,\a);
\draw[line width=2pt] (0,3)--(1,3)--(1,2)--(2,2)--(2,0);
\draw[line width=2pt] (1,5)--(1,4)--(2,4)--(2,3)--(3,3)--(3,2)--(2,2);
\draw[line width=2pt] (5,2)--(4,2)--(4,1)--(3,1)--(3,4)--(5,4);
%%%
\node at (0.5,4.5) {\small 1};
\node at (1.5,4.5) {\small 2};
\node at (2.5,4.5) {\small 3};
\node at (3.5,4.5) {\small 4};
\node at (4.5,4.5) {\small 5};
%%%
\node at (0.5,3.5) {\small 4};
\node at (1.5,3.5) {\small 5};
\node at (2.5,3.5) {\small 1};
\node at (3.5,3.5) {\small 2};
\node at (4.5,3.5) {\small 3};
%%%
\node at (0.5,2.5) {\small 5};
\node at (1.5,2.5) {\small 3};
\node at (2.5,2.5) {\small 2};
\node at (3.5,2.5) {\small 1};
\node at (4.5,2.5) {\small 4};
%%%
\node at (0.5,1.5) {\small 3};
\node at (1.5,1.5) {\small 1};
\node at (2.5,1.5) {\small 4};
\node at (3.5,1.5) {\small 5};
\node at (4.5,1.5) {\small 2};
%%%
\node at (0.5,0.5) {\small 2};
\node at (1.5,0.5) {\small 4};
\node at (2.5,0.5) {\small 5};
\node at (3.5,0.5) {\small 3};
\node at (4.5,0.5) {\small 1};
\end{tikzpicture} 
\end{center}
\caption{A completed Pentadoku puzzle}
\label{fig:pentadoku}
\end{minipage}
\hfill
\begin{minipage}[c]{0.45\linewidth}
\begin{center}
\begin{tikzpicture}[scale=0.8]
\foreach \a in {0,1,...,6}
                \draw (\a,0)--(\a,6);
\foreach \a in {0,1,...,6}
                \draw (0,\a)--(6,\a);
\foreach \a in {0,6}
                \draw[line width=2pt] (\a,0)--(\a,6);
\foreach \a in {0,6}
                \draw[line width=2pt] (0,\a)--(6,\a);
\foreach \a in {2,4}
		\draw[red,line width=2pt] (0,\a)--(6,\a);
\draw[red,line width=2pt] (3,0)--(3,6);
\foreach \a in {2,4}
		\draw[blue,line width=2pt] (\a,0)--(\a,6);
\draw[blue,line width=2pt] (0,3)--(6,3);
%%%
\node at (0.5,5.5) {\small 1};
\node at (1.5,5.5) {\small 2};
\node at (2.5,5.5) {\small 3};
\node at (3.5,5.5) {\small 4};
\node at (4.5,5.5) {\small 5};
\node at (5.5,5.5) {\small 6};
%%%
\node at (0.5,4.5) {\small 4};
\node at (1.5,4.5) {\small 5};
\node at (2.5,4.5) {\small 6};
\node at (3.5,4.5) {\small 1};
\node at (4.5,4.5) {\small 2};
\node at (5.5,4.5) {\small 3};
%%%
\node at (0.5,3.5) {\small 3};
\node at (1.5,3.5) {\small 6};
\node at (2.5,3.5) {\small 2};
\node at (3.5,3.5) {\small 5};
\node at (4.5,3.5) {\small 1};
\node at (5.5,3.5) {\small 4};
%%%
\node at (0.5,2.5) {\small 5};
\node at (1.5,2.5) {\small 1};
\node at (2.5,2.5) {\small 4};
\node at (3.5,2.5) {\small 3};
\node at (4.5,2.5) {\small 6};
\node at (5.5,2.5) {\small 2};
%%%
\node at (0.5,1.5) {\small 2};
\node at (1.5,1.5) {\small 3};
\node at (2.5,1.5) {\small 1};
\node at (3.5,1.5) {\small 6};
\node at (4.5,1.5) {\small 4};
\node at (5.5,1.5) {\small 5};
%%%
\node at (0.5,0.5) {\small 6};
\node at (1.5,0.5) {\small 4};
\node at (2.5,0.5) {\small 5};
\node at (3.5,0.5) {\small 2};
\node at (4.5,0.5) {\small 3};
\node at (5.5,0.5) {\small 1};
\end{tikzpicture} 
\caption{A puzzle with simultaneous box conditions}
\label{fig:overlayed-cages}
\end{center}
\end{minipage}
\end{figure}

Many different box/cage arrangements are possible.  The tilings of a $5 \times 5$ grid by \emph{distinct} pentominoes are enumerated and displayed at \cite{Isomer}.  Setting $n=5$, we computed the $100 \times 100$ matrix $M'$ corresponding to each of these box arrangements.  The nullity of $M'$ was found to equal $27$ or $23$, depending on whether an `I' tile is present or not, respectively.  Curiously, this matches what \eqref{eq:nullity} would produce for rectangular boxes if we were to substitute $h+w=3$ or $2$, respectively.

A generalization which we have not considered could allow two or more simultaneous box patterns.  This is natural because the row and column conditions in a latin square can already be viewed as degenerate box conditions.  An example for $n=6$ with both $2 \times 3$ and $3 \times 2$ boxes is given in Figure~\ref{fig:overlayed-cages}. Using a generalized notion of tile and a suitably enlarged linear system, similar methods as in this paper could apply, at least in principle.

We next make a brief remark about the coherent configuration we used in Section~\ref{sec:aa}.  The number of relations was quite large, and pushed many of our estimates and calculations into computer-assisted territory.  It is possible that the complexity of the algebra can be reduced.  For instance, our matrix $M$ and the eigenprojectors were all symmetric, meaning that relations pair up with their transpose.  Also, every row-column pair uniquely determines a box, so it is plausible that box-symbol edges could be eliminated.  We made an attempt to work with the Schur complement of $M$ relative to box-symbol edges, but some technical challenges discouraged us from going further with that approach.

Methods of algebraic graph theory have been applied to Sudoku before, but in a slightly different way.  A \emph{Sudoku graph} has $n^2$ vertices corresponding to cells, and two vertices are declared adjacent if they share the same row, column or box.  Eigenvalues and eigenvectors of Sudoku graphs have been investigated in \cite{AAS,KSU,Sander}.  Although they have integral eigenvalues and Kronecker-structured $\{\pm 1,0\}$-valued eigenvectors, as ours, we could see no way to use the Sudoku graph alone to build the linear system for completion.  Still, it would be interesting to explore the Sudoku graph in the context of completing partial Sudoku.

As a last remark, our results are only about fractional completion.  For partial latin squares, the iterative absorption methods of \cite{BKLOT} are able to convert a sparseness guarantee for fractional completion into a guarantee for (exact) completion.  We do not know whether these or other methods could work for the Sudoku setting.

\section*{Acknowledgements}

Research of Peter Dukes is supported by NSERC Discovery Grant RGPIN--2017--03891.  
%The authors are grateful to Akihiro Munemasa, who suggested expressing our coherent configuration in terms of a group action.

%%%%%%%%%%%%%%%%

\end{document}